\documentclass[11pt,leqno]{article}
\usepackage{hyperref}
\usepackage{soul}
\usepackage{pdfsync}
\usepackage{listings}
\usepackage{enumitem}
\usepackage[doc]{optional}
\usepackage{footnote}
\usepackage{xcolor}
\definecolor{labelkey}{rgb}{0,0.08,0.45}
\definecolor{rekey}{rgb}{0,0.6,0.0}
\definecolor{Brown}{rgb}{0.45,0.0,0.05}
\usepackage{exscale,relsize}
\usepackage{amsmath}
\usepackage{amsfonts}
\usepackage{amssymb}
\usepackage{calc}
\usepackage{theorem}
\usepackage{pifont}      
\usepackage{graphicx}
\newcommand{\wk}{\ensuremath{\operatorname{w*}}}

\newcommand{\scal}[2]{\langle{{#1},{#2}}\rangle}

\newcommand{\RR}{\ensuremath{\mathbb R}}

\newcommand{\RX}{\ensuremath{\,\left]-\infty,+\infty\right]}}

\newcommand{\NN}{\ensuremath{\mathbb N}}

\newcommand{\thalb}{\ensuremath{\tfrac{1}{2}}}

\newcommand{\menge}[2]{\big\{{#1} \mid {#2}\big\}}
\newcommand{\diff}{\mathrm{d}}
\newcommand{\qede}{\hspace*{\fill}$\Diamond$\medskip}

\newcommand{\To}{\ensuremath{\rightrightarrows}}

\newcommand{\di}{\ensuremath{\operatorname{d}}}

\newcommand{\dom}{\ensuremath{\operatorname{dom}}}

\newcommand{\cont}{\ensuremath{\operatorname{cont}}}
\newcommand{\gra}{\ensuremath{\operatorname{gra}}}
\newcommand{\epi}{\ensuremath{\operatorname{epi}}}

\newcommand{\prox}{\ensuremath{\operatorname{prox}}}

\newcommand{\inte}{\ensuremath{\operatorname{int}}}

\newcommand{\ran}{\ensuremath{\operatorname{ran}}}

\newcommand{\Id}{\ensuremath{\operatorname{Id}}}

\renewcommand{\iff}{\ensuremath{\Leftrightarrow}}
\renewcommand{\phi}{\ensuremath{\varphi}}
\newcommand{\bx}{\bar x}

\newtheorem{theorem}{Theorem}[section]
\newtheorem{lemma}[theorem]{Lemma}
\newtheorem{fact}[theorem]{Fact}
\newtheorem{corollary}[theorem]{Corollary}
\newtheorem{proposition}[theorem]{Proposition}

\theoremstyle{plain}{\theorembodyfont{\rmfamily}
}
\theoremstyle{plain}{\theorembodyfont{\rmfamily}
}
\theoremstyle{plain}{\theorembodyfont{\rmfamily}
}
\theoremstyle{plain}{\theorembodyfont{\rmfamily}
\newtheorem{example}[theorem]{Example}}
\theoremstyle{plain}{\theorembodyfont{\rmfamily}
\newtheorem{remark}[theorem]{Remark}}
\newtheorem{conj}[theorem]{Conjecture}
\newtheorem{problem}[theorem]{Problem}
\theoremstyle{plain}{\theorembodyfont{\rmfamily}
}

\begin{document}

\title{\bf Applications of Convex Analysis within Mathematics}
\author{Francisco J. Arag\'{o}n Artacho\thanks{Centre for Computer Assisted Research Mathematics and its Applications (CARMA), University of Newcastle, Callaghan, NSW 2308, Australia. E-mail:
\texttt{francisco.aragon@ua.es}},\;
Jonathan M.
Borwein\thanks{Centre for Computer Assisted Research Mathematics and its Applications (CARMA), University of Newcastle, Callaghan, NSW 2308, Australia. E-mail:
\texttt{jonathan.borwein@newcastle.edu.au}.
Laureate Professor at the University of Newcastle and Distinguished Professor at  King
Abdul-Aziz University, Jeddah. },\;
Victoria Mart\'{i}n-M\'{a}rquez\thanks{Departamento de An\'{a}lisis Matem\'{a}tico, Universidad de Sevilla, Spain. E-mail:
\texttt{victoriam@us.es}},\\
and Liangjin\
Yao\thanks{Centre for Computer Assisted Research Mathematics and its Applications (CARMA), University of Newcastle, Callaghan, NSW 2308, Australia.
E-mail:  \texttt{liangjin.yao@newcastle.edu.au}.}}

\date{July 19, 2013} \maketitle

\begin{abstract}
In this paper, we study convex analysis and its theoretical
applications. We first apply  important tools of convex analysis to
Optimization and to Analysis.  We then show  various deep applications  of convex
analysis and especially infimal convolution in Monotone Operator Theory. Among other things, we recapture the Minty  surjectivity theorem in Hilbert space, and present a new proof of the sum
theorem in reflexive spaces.  More technically, we also discuss
autoconjugate representers for  maximally monotone operators. Finally,
we consider various other applications in mathematical analysis.
\end{abstract}



\allowdisplaybreaks[1]

\section{Introduction}

While other articles in this collection look at the applications of Moreau's seminal work, we have opted to illustrate the power of his ideas  theoretically within optimization theory and within mathematics more generally.  Space constraints preclude being comprehensive, but we think the presentation made shows how elegantly much of modern analysis can be presented thanks to the work of Jean-Jacques Moreau and others.

\subsection{Preliminaries}

Let $X$ be a real Banach space with norm $\|\cdot\|$ and dual norm $\|\cdot\|_*$. When there is no ambiguity we suppress the $*$.  We write  $X^*$ and $\langle\,\cdot\,,\cdot\,\rangle$ for the real dual space of continuous linear
functions and the duality paring, respectively, and denote the closed unit ball  by  $B_X:= \{x \in
X \mid \|x\| \le 1\}$ and  set $\NN:=\{1,2,3,\ldots\}$. We identify
$X$ with its canonical image in the bidual space $X^{**}$. A set $C \subseteq X$ is said to be {\em convex} if it contains all line segments between its members:
$\lambda x + (1 - \lambda) y \in C$ whenever $x,y \in C$ and $0 \le \lambda \le 1$.

Given a subset $C$ of $X$,
$\inte C$ is the \emph{interior} of $C$ and
$\overline{C}$ is   the
\emph{norm closure} of $C$.
 For a set $D\subseteq X^*$, $\overline{D}^{\wk}$
is the weak$^{*}$ closure of $D$.
The \emph{indicator function} of $C$, written as $\iota_C$, is defined
at $x\in X$ by
\begin{align}
\iota_C (x):=\begin{cases}0,\,&\text{if $x\in C$;}\\
+\infty,\,&\text{otherwise}.\end{cases}\end{align}
The \emph{support function} of $C$, written as $\sigma_C$,
 is defined by $\sigma_C(x^*):=\sup_{c\in C}\langle c,x^*\rangle$.
  There is also a naturally
associated  (metric) {\em distance function}, that is,
\begin{equation}\label{note:distance-function}
\di_C(x) := \inf\left\{\|x - y\| \mid y \in C\right\}.
\end{equation}
Distance functions play a central role in convex analysis, both theoretically and algorithmically.

Let $f\colon X\to \RX$ be a function. Then
$\dom f:= f^{-1}(\RR)$ is the \emph{domain} of $f$, and
the \emph{lower level sets} of a function $f:X \to \RX$ are the sets
$\{x \in X\mid f(x) \le \alpha\}$ where $\alpha \in \RR$.
The \emph{epigraph} of $f$ is $\epi f := \{(x,r)\in
X\times\RR\mid f(x)\leq r\}$.  We will denote the set of points of
continuity of $f$ by $\cont f$. The function $f$ is said to be \emph{convex} if for any $x,y\in\dom f$ and any $\lambda\in[0,1]$, one has
$$f(\lambda x+(1-\lambda)y)\leq\lambda f(x)+(1-\lambda)f(y).$$
We say $f$ is proper if $\dom f\neq\varnothing$.
Let $f$ be proper. The \emph{subdifferential} of
$f$ is defined by
   $$\partial f\colon X\To X^*\colon
   x\mapsto \{x^*\in X^*\mid \scal{x^*}{y-x}\leq f(y)-f(x), \text{ for all }y\in X\}.$$
By the definition of $\partial f$,  even when
$ x \in \dom f$, it is possible that $\partial f( x)$ may be empty. For example $\partial f(0)=\varnothing$ for $f(x):=-\sqrt{x}$ whenever $x\geq 0$ and $f(x):=+\infty$ otherwise.
If $x^* \in \partial f(
x)$ then $x^*$ is said to be a {\em subgradient} of $f$ at $
x$.
An important example of a subdifferential is the {\em normal cone} to a convex set $C \subseteq X$ at a point $x \in C$ which is defined as $N_C(x):=\partial \iota_C(x)$.

Let $g\colon X\rightarrow\RX$.
Then the \emph{inf-convolution} $f\Box g$
is the function defined on $X$ by
\begin{equation*}f\Box g\colon
x \mapsto \inf_{y\in X}
\big\{f(y)+g(x-y)\big\}.
\end{equation*}
(In~\cite{M63} Moreau studied   inf-convolution when $X$ is an arbitrary commutative semigroup.) Notice that, if both $f$ and $g$ are convex, so it is $f\Box g$ (see, e.g.,~\cite[p.~17]{Mnotes}).

We use the convention that $(+\infty)+(-\infty)=+\infty$
 and $(+\infty)-(+\infty)=+\infty$.
We will say a function $f:X \to \RX$ is {\em Lipschitz on a subset
$D$} of $\dom f$ if there is a constant $M \ge 0$ so that $|f(x) - f(y)|
\le M\|x - y\|$ for all $x,y \in D$. In this case $M$ is said to be a {\em Lipschitz
constant} for $f$ on $D$. If for each $x_0 \in D$, there is an open
set $U \subseteq D$ with $x_0 \in U$ and a constant $M$ so that $|f(x) - f(y)| \le M\|x - y\|$
for all $x,y \in U$, we will say $f$ is {\em locally Lipschitz on
$D$}. If $D$ is the entire space, we simply say $f$ is Lipschitz or
locally Lipschitz respectively.

Consider a function $f:X \to \RX$; we  say $f$ is
\emph{lower-semicontinuous}\ (lsc) if $\liminf_{x \to \bx} f(x) \ge
f(\bx)$ for all $\bx \in X$, or equivalently, if $\epi f$ is closed.
The function $f$ is said to be \emph{sequentially weakly lower
semi-continuous} if for every $\bx\in X$ and every sequence
$(x_n)_{n\in \NN}$ which is weakly convergent to $\bx$, one has
$\liminf_{n \to \infty} f(x_n) \ge f(\bx)$. This is a useful distinction since there are infinite dimensional Banach spaces (\emph{Schur spaces} such as $\ell^1$) in which weak and norm convergence coincide for sequences, see \cite[p. 384, esp. Thm 8.2.5]{BorVan}.

\subsection{Structure of this paper}

The remainder of this paper is organized as follows. In
Section~\ref{s:aux}, we describe results about Fenchel conjugates and the subdifferential operator, such as Fenchel duality, the Sandwich theorem, etc. We also look at some interesting convex functions and inequalities. In Section~\ref{SecChev:1}, we discuss the Chebyshev problem from abstract approximation.
In Section \ref{SecMon:1}, we show applications of convex analysis in Monotone Operator Theory.
We reprise such results as
the Minty surjectivity theorem, and present a new proof of the sum theorem in reflexive spaces.  We also discuss
 Fitzpatrick's problem on so called autoconjugate representers for  maximally monotone operators. In Section \ref{Sec:Other} we discuss various other applications.

\section{Subdifferential operators, conjugate functions \& Fenchel duality}\label{s:aux}

We begin with some fundamental properties of convex sets and convex functions. While many results hold in all locally convex spaces, some of the most important such as \ref{basic-prin:E4}(b) in the next Fact do not.

\begin{fact}\emph{(Basic properties~{\cite[Ch. 2 and 4]{BorVan}}.)}\label{basic-prin}
The following hold.
\begin{enumerate}
\item\label{basic-prin:E1} The (lsc) convex functions form a convex cone closed under pointwise suprema: if  $f_{\gamma}$
is convex (and lsc) for each $\gamma \in \Gamma$  then so is $x \mapsto\sup_{\gamma \in \Gamma} f_{\gamma}(x)$.

\item A function $f$ is convex if and only if $\epi f$ is convex if and only if $\iota_{\epi f}$
is convex.

\item Global minima and local minima in the domain coincide for proper convex functions.

\item \label{basic-prin:E4} Let $f$ be a proper convex function and let $x\in\dom f$. (a) $f$ is locally Lipschitz at $x$ if and only $f$ is continuous at $x$ if and only if
 $f$ is locally bounded at $x$. (b) Additionally, if $f$ is lower semicontinuous, then $f$ is continuous at every point in $\inte\dom f$.

\item \label{basic-prin:E5}A proper  lower semicontinuous and convex function
is bounded
from below by a continuous affine function.

\item If $C$ is a nonempty set, then $\di_C(\cdot)$ is non-expansive (i.e., is  a Lipschitz function with constant one). Additionally, if $C$ is convex, then
$\di_C(\cdot)$ is a convex function.

\item If $C$ is a convex set, then
$C$ is weakly closed if and only if it is norm closed.
\item \label{three-slope}\emph{Three-slope inequality}: Suppose $f:\RR\to]-\infty,\infty]$ is convex and $a<b<c$. Then
$$\frac{f(b)-f(a)}{b-a}\leq\frac{f(c)-f(a)}{c-a}\leq \frac{f(c)-f(b)}{c-b}.$$
\end{enumerate}
%
%
\end{fact}

The following trivial fact shows the fundamental significance of subgradients in optimization.

\begin{proposition} [Subdifferential at optimality] \label{1.4.3}
Let $f \colon X \rightarrow \RX$ be a proper convex function.
Then the point $\bx\in \dom f$ is a (global)
minimizer of $f$ if and only if $0 \in \partial f(\bx)$.
\end{proposition}

The {\em directional derivative} of
$f$ at $\bar x \in \dom f$ in the direction $d$ is defined by
\label{note:directional-derivative}
$$
f^\prime(\bar x;d) := \lim_{t \to 0^+} \frac{f(\bar x + td) - f(\bar
x)}{t}
$$
if the limit exists. If $f$ is convex, the directional derivative is everywhere finite at any point of $\inte\dom f$, and it turns out to be Lipschitz at $\cont f$. 
We use the term directional derivative with the understanding that it is actually
a {\em one-sided} directional derivative.

If the directional derivative $f'(\bar x,d)$ exists for all directions $d$ and the operator $f'(\bx)$ defined by $\langle f'(\bx),\cdot\,\rangle:=f'(\bx;\cdot\,)$ is linear and bounded, then we say that $f$ is {\em G\^ateaux differentiable} at $\bar x$, and $f'(\bx)$ is called the \emph{G\^ateaux derivative}. Every function $f:X \to \RX$ which is
lower semicontinuous, convex and G\^ateaux differentiable at $x$, it is continuous at $x$. Additionally, the following properties are relevant for the existence and uniqueness of the subgradients.

\begin{proposition}\emph{(See {\cite[Fact 4.2.4 and Corollary 4.2.5]{BorVan}}.)}
\label{GatPC} Suppose $f: X \to \RX$ is convex.
\begin{enumerate}
\item  If $f$ is G\^ateaux differentiable at
$\bx$, then $f^\prime(\bx) \in \partial f(\bx)$.
\item If $f$ is continuous at $\bx$,
then $f$ is G\^ateaux differentiable at $\bx$ if and only if
$\partial f(\bx)$ is a singleton.
\end{enumerate}
\end{proposition}
\begin{example} We show that part (ii) in Proposition~\ref{GatPC} is not always true  in infinite dimensions without  continuity hypotheses.

\begin{enumerate}
\item [(a)] The indicator of the Hilbert cube $C:=\{x=(x_1,x_2,\ldots) \in \ell^2: |x_n| \le 1/n, \forall n\in\NN\}$ at zero or any other non-support point has a unique subgradient but is nowhere G\^ateaux differentiable.

\item [(b)]  \emph{Boltzmann-Shannon entropy}  $x \mapsto \int_0^1 x(t)\log(x(t)){\rm d} t$ viewed as a lower semicontinuous and  convex function on $L^1[0,1]$ has unique subgradients at $x(t)>0$ a.e. but is nowhere G\^ateaux differentiable (which for a lower semicontinuous and convex function in Banach space implies continuity).
    \end{enumerate}
  That  G\^ateaux differentiability of a convex and lower semicontinuous function implies continuity at the point is a consequence of the Baire category theorem. \qede
\end{example}

The next result proved by Moreau in 1963 establishes the relationship between subgradients and directional derivatives, see also~\cite[page~65]{Mnotes}.  Proofs can be also found in most of the books in variational analysis, see e.g.~\cite[Theorem 4.2.7]{BZ05}.

\begin{theorem}[Moreau's max formula~\cite{M63_2}]  \label{thm:max-formula}
Let
$f \colon X \rightarrow \RX$ be a  convex function and let $d \in X$.
Suppose that $f$ is continuous at $\bar x $. Then, $\partial f(\bar x) \ne \varnothing$ and
\begin{eqnarray}
f'(\bar x;d) = \max\{\langle x^*,d\rangle \mid x^* \in \partial f(\bar x)\}.
\end{eqnarray}
\end{theorem}

\medskip

Let $f:X\rightarrow[-\infty,+\infty]$.
The {\em Fenchel conjugate} (also called the {\em Legendre-Fenchel
conjugate}\footnote{Originally the connection was made between a monotone function on an interval and its inverse. The convex functions then arise by integration.} or {\em transform}) of $f$ is the function \label{note:Fenchel-conjugate}
$f^*:X^* \to [-\infty,+\infty]$ defined by
$$
f^*(x^*) := \sup_{x \in X} \{\langle x^*,x \rangle - f(x) \}.
$$
We can also consider the conjugate of $f^*$ called
the {\em biconjugate} of $f$ and denoted by $f^{**}$.
This is a convex function on $X^{**}$ satisfying $f^{**}|_X\leq f$. A useful and instructive example is $\sigma_C = \iota_C^*$.

\begin{example}\label{ex:conj_norm}
Let $1 < p < \infty$ . If $f(x):=\frac{\|x\|^p}{p}$ for $x\in X$ then $f^*(x^*)=\frac{\|x^*\|_*^q}{q}$, where $\frac{1}{p}+\frac{1}{q}=1$.
Indeed, for any $x^*\in X^*$, one has
\begin{align*}
f^*(x^*)=\sup_{\lambda\in\RR_+}\sup_{\|x\|=1}\left\{\langle x^*,\lambda x\rangle -\frac{\|\lambda x\|^p}{p}\right\}=\sup_{\lambda\in\RR_+}\left\{\lambda \|x^*\|_*-\frac{\lambda^p}{p}\right\}=\frac{\|x^*\|_*^q}{q}.
\end{align*}
\qede
\end{example}

By direct construction and Fact \ref{basic-prin} \ref{basic-prin:E1}, for any function $f$, the  conjugate function $f^*$ is  always convex and lower semicontinuous, and if the domain of
$f$ is nonempty, then $f^*$ never takes the value $-\infty$. The conjugate plays a role in
convex analysis in many ways analogous to the role played by the
Fourier transform in harmonic analysis with infimal convolution, see below, replacing integral convolution and sum replacing product  \cite[Chapter 2.]{BorVan}.

\subsection{Inequalities and their applications}

 An immediate  consequence of the definition is that for $f,g:X \to
[-\infty,+\infty]$, the inequality $f \ge g$ implies $f^* \le g^*$.
 An important result which is straightforward to prove is the following.

\begin{proposition}[Fenchel--Young] \label{prop:Fenchel-Young-fd}
 Let $f:X \to
\RX$. All
points $x^* \in X^*$ and $x \in \dom f$ satisfy the inequality
\begin{equation}
f(x) + f^*(x^*) \ge \langle x^*,x\rangle.
\label{eq:Fenchel-Young-ineq-fd}
\end{equation}
Equality holds if and only if $x^* \in \partial f(x)$.
\end{proposition}

\begin{example}[Young's inequality]\label{prop:Young}
By taking $f$ as in Example~\ref{ex:conj_norm}, one obtains  directly  from Proposition~\ref{prop:Fenchel-Young-fd}
$$\frac{\|x\|^p}{p}+\frac{\|x^*\|_*^q}{q}\geq\langle x^*,x\rangle,$$
for all $x\in X$ and $x^*\in X^*$, where $p>1$ and $\frac{1}{p}+\frac{1}{q}=1$. When $X=\RR$ one recovers the original Young inequality.\qede
\end{example}

This in turn leads to one of the workhorses of modern analysis:

\begin{example}[H\"{o}lder's inequality]\label{prop:Holder}
Let $f$ and $g$ be measurable on a measure space $(X,\mu)$. Then
\begin{equation}\label{eq:Holder}
\int_X fg\ \mathrm{d}\mu\leq\|f\|_p\|g\|_q,
\end{equation}
where $1 < p < \infty$ and $\frac{1}{p}+\frac{1}{q}=1$. Indeed, by rescaling, we may assume without loss of generality that $\|f\|_p=\|g\|_q=1$. Then Young's inequality in Example~\ref{prop:Young} yields
$$|f(x)g(x)|\leq\frac{|f(x)|^p}{p}+\frac{|g(x)|^q}{q}\quad\text{for } x\in X,$$
and~\eqref{eq:Holder} follows by integrating both sides. The result holds true in the limit for $p=1$ or $p=\infty$.\qede
\end{example}

We next take a brief excursion into special function theory and normed space geometry to emphasize that ``convex functions are everywhere."

\begin{example}[Bohr--Mollerup theorem]\label{ex:bohr}
The \emph{Gamma function} defined for $x>0$ as
$$\Gamma(x):=\int_0^\infty e^{-t}t^{x-1}\diff t=\lim_{n\to\infty}\frac{n!\,n^x}{x(x+1)\cdots(x+n)}$$
is the unique function $f$ mapping the positive half-line to itself and such that (a) $f(1)=1$, (b) $xf(x)=f(x+1)$ and (c) $\log f$ is a convex function.

Indeed, clearly $\Gamma(1)=1$, and it is easy to prove (b) for $\Gamma$ by using integration by parts. In order to show that $\log \Gamma$ is convex, pick any $x,y>0$ and $\lambda\in(0,1)$ and apply H\"{o}lder's inequality~\eqref{eq:Holder} with $p=1/\lambda$ to the functions $t\mapsto e^{-\lambda t}t^{\lambda(x-1)}$ and $t\mapsto e^{-(1-\lambda)t}t^{(1-\lambda)(y-1)}$. For the converse, let $g:=\log f$. Then (a) and (b) imply
$g(n+1+x)=\log\left[x(1+x)\ldots (n+x) f(x)\right]$ and thus $g(n+1)=\log(n!)$. Convexity of $g$ together with the three-slope inequality, see Fact~\ref{basic-prin}\ref{three-slope}, implies that
$$g(n+1)-g(n)\leq\frac{g(n+1+x)-g(n+1)}{x}\leq g(n+2+x)-g(n+1+x),$$
and hence,
$$x\log(n)\leq \log\left(x(x+1)\cdots(x+n)f(x)\right)-\log(n!)\leq x\log(n+1+x);$$
whence,
$$0\leq g(x)-\log\left(\frac{n!\,n^x}{x(x+1)\cdots(x+n)}\right)\leq x\log\left(1+\frac{1+x}{n}\right).$$
Taking limits when $n\to\infty$ we obtain
$$f(x)=\lim_{n\to\infty}\frac{n!\,n^x}{x(x+1)\cdots(x+n)}=\Gamma(x).$$
As a bonus we recover a classical and important limit formula for $\Gamma(x)$.

  Application of the Bohr--Mollerup theorem is often \emph{automatable}
in a computer algebra system, as we now  illustrate.
 Consider the \emph{beta function}
\begin{eqnarray}
\beta(x,y) &:=& \int_0^1 t^{x-1}(1-t)^{y-1}\di t \label{beta-def}
\end{eqnarray}
\noindent for $\mathrm{Re}(x), \mathrm{Re}(y) > 0$. As is often
established using polar coordinates and double integrals
\begin{eqnarray}
\beta(x,y) &=& \frac{\Gamma(x)\,\Gamma(y)}{\Gamma(x+y)}.
\label{eq:bint}
\end{eqnarray}
We may use the Bohr--Mollerup theorem with
$$
f:=x \to \beta(x,y)\,\Gamma(x+y)/\Gamma(y)
$$
to prove (\ref{eq:bint}) for real $x, y$.

Now (a) and (b) from
Example~\ref{ex:bohr} are easy to verify. For (c) we again use H\"older's inequality
to show $f$ is log-convex. Thus, $f=\Gamma$ as required.
\qede
\end{example}

\begin{example}[Blaschke--Santal\'o theorem]
\label{ex:Blaschke-theorem}
The volume of a unit ball in the $\|\cdot\|_p$-norm,
$V_n(p)$ is
\begin{eqnarray}
V_n(p) \ &=& \ 2^n \frac{\Gamma(1+\frac{1}{p})^n}{\Gamma(1+\frac{n}{p})}. \label{volp}
\end{eqnarray}
as was first determined by Dirichlet. When $p=2$, this gives
\begin{eqnarray*}
V_n &=& 2^n \frac{\Gamma(\frac 32)^n }{\Gamma(1+\frac n2)} \; = \;
\frac{\Gamma(\frac 12)^n}{\Gamma(1+\frac n2)},
\end{eqnarray*}
which is more concise than that usually recorded in texts.

 Let $C$ in $\RR^n$  be  a {\em convex body} which is symmetric around zero,
that is,  a closed bounded convex set with nonempty interior.
Denoting $n$-dimensional Euclidean volume of $S\subseteq \RR^n$ by
$V_n(S)$, the {\em Blaschke--Santal\'o} inequality says
\begin{equation} \label{eq:BS}
V_n(C)\,V_n(C^\circ) \ \leq \ V_n(E)\,V_n(E^\circ) \ =\
V_n^2(B_n(2))
\end{equation}
where maximality holds (only)  for  \emph{any symmetric} ellipsoid $E$ and
$B_n(2)$ is the Euclidean  unit ball. It is \emph{conjectured}  the minimum
is attained by the 1-ball and the $\infty$-ball. Here as always the
polar set is defined by $C^\circ \ := \ \{y\in \RR^n \colon \langle
y,x\rangle \le 1 \mbox{ for all } x\in C\}.$

The $p$-ball case of (\ref{eq:BS}) follows by proving the
following convexity result:

\begin{theorem}[Harmonic-arithmetic log-concavity]  \emph{The function
\[V_\alpha(p):=2^\alpha {\Gamma\left(1+\frac{1}{p}\right)^\alpha/
\Gamma\left(1+\frac{\alpha}{p}\right)}\] satisfies
\begin{equation} \label{eqV}
V_\alpha(p)^\lambda\,V_\alpha(q)^{1-\lambda}\ < \
V_\alpha\left(\frac{1}{\frac{\lambda}{p}+\frac{1-\lambda}{q}}\right),
\end{equation}
 for all $\alpha>1$,
if $p,q>1$, $p\ne q$, and $\lambda\in(0,1)$.}
\end{theorem}
Set $\alpha:=n$, $\frac1p+\frac1q=1$  with $\lambda=1-\lambda
=1/2$ to recover the $p-$norm case of the Blaschke--Santal\'o
inequality. It is amusing to deduce the corresponding lower bound.
This technique extends to  various substitution norms. Further details may be found in~\cite[\S5.5]{BorBai08}. Note that we may easily explore $V_\alpha(p)$ graphically. \qede
\end{example}

\subsection{The biconjugate and duality}

The next result has been associated by different authors with the names of Legendre, Fenchel, Moreau and H\"{o}rmander; see, e.g., \cite[Proposition 4.4.2]{BorVan}.

\begin{proposition}\rm(\textbf{H\"{o}rmander}\footnote{H\"{o}rmander first proved the case of support and indicator functions in \cite{H55} which led to discovery of general result.})\emph{(See
\cite[Theorem~2.3.3]{Zalinescu} or \cite[Proposition~4.4.2(a)]{BorVan}.)}
Let  $f: X \to \RX$ be a proper function. Then
$$f\text{ is convex and lower semicontinuous }\iff f=f^{**}|_X.$$
\end{proposition}

\begin{example}[Establishing convexity]{(See
\cite[Theorem~1]{BorIN}.)}
We may compute conjugates by hand or using the software \emph{SCAT} \cite{BorHam}. This is discussed further in Section \ref{ssec:CAS}. Consider $f(x):=e^x$. Then $f^*(x)=x\log(x)-x$ for $x \ge 0$ (taken to be zero at zero) and is infinite for $x<0$. This establishes the convexity of $x\log(x)-x$ in a way that takes no knowledge of $x\log(x)$.

A more challenging case is the following (slightly corrected) conjugation formula \cite[p. 94, Ex.~13]{BorLew} which can be computed algorithmically:
Given real $\alpha_1,\alpha_2, \ldots ,\alpha_m > 0$, define $\alpha := \sum_i
\alpha_i$ and suppose a real $\mu$ satisfies $\mu > \alpha + 1$.  Now define a
function
$f:\RR^m \times \RR \mapsto \RX$ by
\begin{align*}
f(x,s) :=\begin{cases}
\mu^{-1}s^\mu \prod_i x_i^{-\alpha_i}  &  \mbox{if}~x \in \RR^m_{++},~s \in \RR_+; \\ \\
0 &  \mbox{if}~\exists x_i=0,\, x\in\RR^m_+, ~s=0; \\
+\infty                                &  \mbox{otherwise}.
\end{cases},\quad\forall x:=(x_n)^m_{n=1}\in\RR^m,\, s\in\RR.
\end{align*}
It transpires that
\begin{align*}
f^*(y,t) = \begin{cases}
\rho \nu^{-1} t^\nu \prod_i (-y_i)^{-\beta_i}
                               &  \mbox{if}~y \in \RR^m_{--},~t \in \RR_+  \\
 0              &  \mbox{if}~y \in \RR^m_{-},~t  \in \RR_- \\
+\infty                  &  \mbox{otherwise}
\end{cases},\quad\forall y:=(y_n)^m_{n=1}\in\RR^m,\, t\in\RR.
\end{align*}
for constants
\[
\nu := \frac{\mu}{\mu - (\alpha + 1)},~~
\beta_i := \frac{\alpha_i}{\mu - (\alpha + 1)},~~
\rho := \prod_i \Big( \frac{\alpha_i}{\mu} \Big)^{\beta_i}.
\]
We deduce  that $f=f^{**}$, whence $f$ (and $f^*$) is (essentially strictly) convex.
For attractive alternative proof of convexity see \cite{Mar}. Many other substantive examples are to be found in \cite{BorLew,BorVan}.\qede
\end{example}

The next theorem gives us a remarkable sufficient condition for convexity of functions in terms of the G\^ateaux differentiability of the conjugate. There is a simpler analogue for the Fr\'echet derivative.

\begin{theorem}\emph{(See {\cite[Corollary 4.5.2]{BorVan}}.)}\label{th:convex_conj_suff}
Suppose $f:X \to \RX$ is such that $f^{**}$ is proper. If $f^*$ is G\^ateaux differentiable at all $x^* \in \dom\partial f^*$ and
$f$ is sequentially weakly lower semicontinuous, then $f$ is convex.
\end{theorem}

Let $f:X \to \RX$. We say $f$ is {\em coercive} if $\lim_{\|x\| \to
\infty} f(x) = +\infty$. We say $f$ is {\em supercoercive} if $\lim_{\|x\| \to
\infty} \frac{f(x)}{\|x\|} = +\infty$.

\begin{fact}\emph{(See \cite[Fact 4.4.8]{BorVan}.)}\label{CorEquiv}
If $f$ is proper
convex and lower semicontinuous at some point in
its domain, then the following statements are equivalent.
\begin{enumerate}
\item $f$ is coercive.
\item There exist $\alpha > 0$ and $\beta \in \RR$ such that
$f \ge \alpha\|\cdot\| + \beta$.

\item $\liminf_{\|x\| \to \infty} f(x)/\|x\| > 0$.

\item $f$ has bounded lower level sets.
\end{enumerate}
\end{fact}

Because a convex function is continuous at a point if and only if
it is bounded above on a neighborhood of that point (Fact~\ref{basic-prin}\ref{basic-prin:E4}),
we get the following result; see also~\cite[Theorem~7]{H55} for the case of
the indicator function of a bounded convex set.

\begin{theorem}[H\"{o}rmander--Moreau--Rockafellar] \label{Moreau-Rockafellar2}
Let $f:X \to \RX$ be convex and lower semicontinuous at some point in its domain, and let
$x^* \in X^*$. Then $f - x^*$ is coercive if and only if
$f^*$ is continuous at $x^*$.
\end{theorem}

\begin{proof}
``$\Rightarrow$":  By Fact~\ref{CorEquiv},  there exist $\alpha > 0$ and $\beta \in \RR$ such that
$f \ge x^*+\alpha\|\cdot\| + \beta$.  Then $f^*\leq-\beta+\iota_{\{x^*+\alpha B_{X^*}\}}$, from where $x^*+\alpha B_{X^*}\subseteq\dom f^*$. Therefore,  $f^*$ is continuous at $x^*$ by Fact~\ref{basic-prin}\ref{basic-prin:E4}.

``$\Leftarrow$":  By the assumption, there exists $\beta\in\RR$ and $\delta>0$ such that
\begin{align*}
f^*(x^*+z^*)\leq \beta,\quad\forall z^*\in\delta B_{X^*}.
\end{align*}
Thus, by Proposition~\ref{prop:Fenchel-Young-fd},
$$\langle x^*+z^*, y\rangle-f(y) \leq \beta,\quad\forall z^*\in\delta B_{X^*},\,\forall y\in X;$$
whence, taking the supremum with $z^*\in\delta B_{X^*}$,
$$\delta\|y\|-\beta\leq f(y)-\langle x^*, y\rangle,\quad\forall y\in X.$$
Then, by Fact~\ref{CorEquiv}, $f-x^*$ is coercive.
\end{proof}

\begin{example}\label{ex:Moreau-Rockafellar}
Given a set $C$ in $X$, recall that the negative polar cone of $C$
is the convex cone
$$C^-:=\{x^*\in X^*\mid \sup \langle x^*, C\rangle\leq0\}.$$
Suppose that $X$ is reflexive and let $K\subseteq X$ be a closed convex cone. Then $K^-$ is another nonempty
closed convex cone with $K^{--} :=(K^-)^-=K$. Moreover, the indicator function
of $K$ and $K^-$ are conjugate to each other. If we set
$f:=\iota_{K^-}$, the indicator function of the negative polar cone
of $K$, Theorem \ref{Moreau-Rockafellar2} applies to get that
\begin{center}
$x\in \inte K$ if and only if the set $\{x^*\in K^- \mid \langle
x^*,x\rangle\geq \alpha\}$ is bounded for any $\alpha\in\RR$.
\end{center}
Indeed, since $x\in \inte K=\inte \dom \iota^*_{K^-}$ if
and only if $\iota^*_{K^-}$ is continuous at $x$, from Theorem
\ref{Moreau-Rockafellar2} we have that this is true if and only if
the function $\iota_{K^-}-x$ is coercive. Now, Fact
\ref{CorEquiv} assures us that coerciveness is equivalent to
boundedness of the lower level sets, which implies the assertion. \qede
\end{example}

\begin{theorem}[Moreau--Rockafellar duality \cite{Mor64}] \label{Moreau}
Let $f:X \to (-\infty,+\infty]$ be a lower semicontinuous  convex function.
Then $f$ is continuous at $0$ if and only if $f^*$ has weak$^*$-compact lower level sets.
\end{theorem}

\begin{proof} Observe that $f$ is continuous at $0$ if and only if $f^{**}$ is continuous at
$0$ (\cite[Fact~4.4.4(b)]{BorVan})if and only if $f^*$ is coercive (Theorem~\ref{Moreau-Rockafellar2})
if and only if $f^*$ has bounded lower level sets (Fact~\ref{CorEquiv}) if and only
if $f^*$ has weak$^*$-compact lower level sets by
the Banach-Alaoglu theorem (see \cite[Theorem~3.15]{Rudin}). \end{proof}

\begin{theorem}[Conjugates of supercoercive functions] \label{f-superc-iff-f*-bndonbnd}
Suppose $f: X \to \RX$ is a
lower semicontinuous and proper convex function. Then

\begin{enumerate}
\item[(a)] $f$ is supercoercive if and only if $f^*$ is bounded  (above) on bounded sets.

\item[(b)] $f$ is bounded (above) on bounded sets if and only if $f^*$ is supercoercive.
\end{enumerate}
\end{theorem}

\begin{proof} (a) ``$\Rightarrow$'': Given any $\alpha > 0$,
there exists $M$ such that $f(x) \ge \alpha \|x\|$ if $\|x\| \ge M$. Now there exists
$\beta \ge 0$ such that $f(x) \ge -\beta$ if $\|x\| \le M$ by Fact~\ref{basic-prin}\ref{basic-prin:E5}. Therefore
$f \ge \alpha\|\cdot\| + (-\alpha M-\beta)$.
Thus, it implies that $f^*\leq \alpha(\|\cdot\|)^*(\frac{\cdot}{\alpha}) + \alpha M+\beta$
and hence
$f^* \le \alpha M+\beta$ on $\alpha B_{X^*}$.

``$\Leftarrow$'': Let $\gamma > 0$. Now there exists $K$ such that
$f^* \le K$ on $\gamma B_{X^*}$. Then $f \ge \gamma \|\cdot\| - K$ and so
$\liminf_{\|x\| \to \infty} \frac{f(x)}{\|x\|} \ge \gamma$.  Hence $\liminf_{\|x\| \to \infty} \frac{f(x)}{\|x\|}=+\infty$.

(b): According to (a), $f^*$ is supercoercive if and only if $f^{**}$
is bounded on bounded sets. By
\cite[Fact~4.4.4(a)]{BorVan} this  holds if and only if $f$ is bounded (above) on bounded sets. \end{proof}

\bigskip

We finish this subsection by recalling some properties of infimal convolutions. Some of their many applications include
smoothing techniques and approximation. We shall meet them again in Section \ref{SecMon:1}. Let $f,g:X\rightarrow\RX$.
Geometrically, the infimal convolution of $f$ and $g$ is the
largest extended real-valued function whose epigraph contains
the sum of epigraphs of $f$ and $g$ (see example in Figure~\ref{fig:convolution}), consequently it is a convex function.
The following is a useful result concerning the conjugate of the infimal convolution.

\begin{fact}\emph{(See \cite[Lemma 4.4.15]{BorVan} and~\cite[pp.~37-38]{Mnotes}.)} \label{lem:conjugate-convo} If $f$ and $g$ are proper functions on $X$, then $(f \Box g)^* = f^* + g^*$. Additionally, suppose $f,g$ are convex and bounded below. If  $f:X \to \RR$ is continuous (resp. bounded on bounded sets, Lipschitz), then $f \Box g$ is a convex function that is continuous (resp. bounded on bounded sets, Lipschitz).
\end{fact}

\begin{remark}
Suppose $C$ is a nonempty convex set. Then $\di_C = \|\cdot\|\Box \iota_C$, implying that $\di_C$ is
a  Lipschitz convex function.\qede
\end{remark}

\begin{example}
Consider $f,g:\RR\to \RX$ given by
$$f(x):=\left\{\begin{array}{ll}
-\sqrt{1-x^2},& \text{for } -1\leq x\leq 1,\\
+\infty& \text{otherwise,}
\end{array}\right. \quad and \quad g(x):=|x|.$$
The infimal convolution of $f$ and $g$ is
$$(f\Box g)(x)=\left\{\begin{array}{ll}
-\sqrt{1-x^2},& -\frac{\sqrt{2}}{2}\leq x\leq -\frac{\sqrt{2}}{2};\\
|x|-\sqrt{2}, & \text{otherwise.}
\end{array}\right.,$$
as shown in  Figure~\ref{fig:convolution}.\qede
\end{example}

\begin{figure}[ht!]
\centering
\includegraphics[width=.6\textwidth]{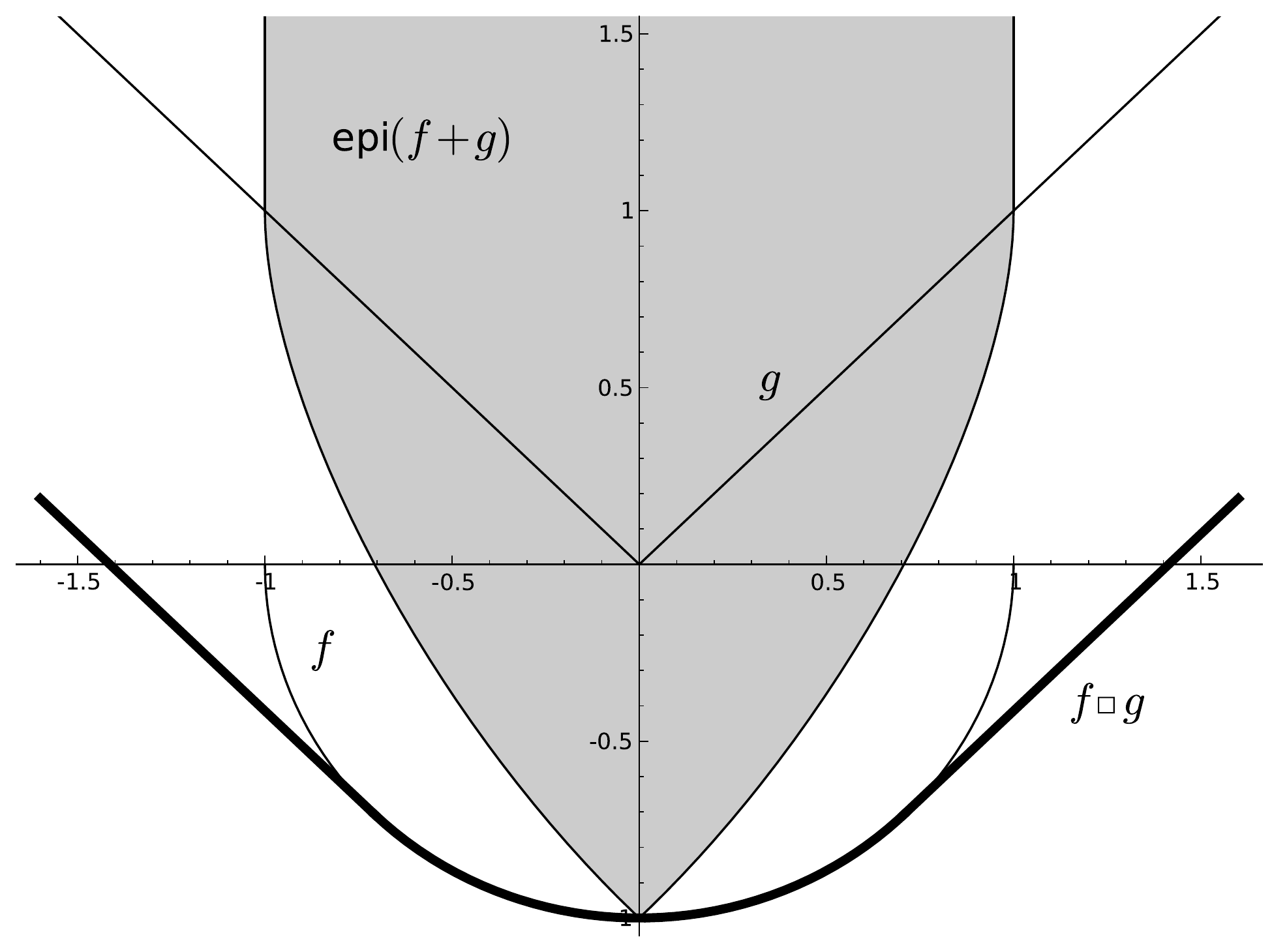}
\caption{Infimal convolution of $f(x)=-\sqrt{1-x^2}$ and $g(x)=|x|$.}\label{fig:convolution}
\end{figure}

\subsection{The Hahn-Banach circle}

Let $T:X \to Y$ be a linear mapping between two Banach spaces $X$ and $Y$. The {\em adjoint} of $T$
is the linear mapping $T^*:Y^*\to X^*$
defined, for $y^*\in Y^*$, by
$$\langle T^*y^*,x\rangle = \langle y^*, Tx\rangle\quad \text{for all } x \in X.$$
A flexible modern version of Fenchel's celebrated duality theorem is:

\begin{theorem}[Fenchel duality] \label{thm:Fenchel-duality-fd}
Let $Y$ be another Banach space,
let $f \colon X \rightarrow \RX$ and $g \colon Y \rightarrow \RX$ be convex functions
and let $T \colon X \rightarrow Y$ be a bounded linear
operator. Define the primal and dual
values $p,d \in [-\infty,+\infty]$
by solving the
{\rm Fenchel problems}
\begin{align}
p &:= \inf_{x \in X} \{f(x)+g(Tx)\} \nonumber
\\
d &:= \sup_{y^* \in Y^*} \{-f^{*}(T^{*}y^*) - g^{*}(-y^*)\}.
\label{eq:Fenchel_Dual_P}
\end{align}
Then these values satisfy the {\rm weak duality}
inequality $p \geq d$.

Suppose further that
$f$, $g$ and $T $satisfy either
\begin{equation}
\bigcup_{\lambda>0} \lambda\left[\dom g - T\dom f\right]
=Y\,\ \mbox{and both}\ f \mbox{ and } g \mbox{ are lower semicontinuous}, \label{eq:Fen-dual-1}
\end{equation}
or the condition
\begin{equation}
 \cont g \cap T\dom f \ne \varnothing. \label{eq:Fen-dual-2}
\end{equation}
Then $p=d$, and the supremum in the dual problem \eqref{eq:Fenchel_Dual_P}
is attained
when finite.  Moreover, the perturbation function $h(u):=\inf_x f(x)+g(Tx+u)$ is convex and continuous at zero.
\end{theorem}

Generalizations of Fenchel duality Theorem can be found in
\cite{BotWank1,BotGradWank}.
  An easy consequence is:

\begin{corollary}[Infimal convolution]\label{infc}  Under the hypotheses of the  Fenchel duality theorem \ref{thm:Fenchel-duality-fd}  $(f + g)^*(x^*)=(f^* \Box g^*)(x^*)$ with attainment when finite.\end{corollary}

Another nice consequence of Fenchel duality is the ability to obtain primal solutions from
dual ones, as we now record.

\begin{corollary}
Suppose the conditions for equality in the Fenchel duality Theorem~\ref{thm:Fenchel-duality-fd} hold,
and that $\bar y^* \in Y^*$ is an optimal dual solution. Then the point
$\bar x \in X$ is optimal for the primal problem if and only if it
satisfies the two conditions $T^* \bar
y^* \in \partial f(\bar x)$ and $-\bar y^*\in \partial g(T\bar x)$.
\end{corollary}

The regularity conditions in Fenchel duality theorem can be weakened
when each function is \emph{polyhedral}, i.e., when their epigraph
is polyhedral.

\begin{theorem}[Polyhedral Fenchel duality]\emph{(See \cite[Corollary~5.1.9]{BorLew}.)}\label{thm:Fenchel-duality-poly} Suppose that $X$ is a finite-dimensional space.
The conclusions of the Fenchel duality
Theorem~\ref{thm:Fenchel-duality-fd} remain valid if the regularity
condition~\eqref{eq:Fen-dual-1} is replaced by the assumption that the functions $f$ and $g$ are
polyhedral with
$$\dom g \cap T \dom f \neq\varnothing.$$
\end{theorem}

Fenchel duality applied to a linear programming program yields the well-known Lagrangian duality.

\begin{corollary}[Linear programming duality] \label{cor:FLC}
Given $c\in \RR^n$, $b\in\RR^m$ and $A$ an  $m\times n$ real matrix, one has
\begin{equation}\label{eq:fenchel_LP}
\inf_{x \in \RR^n} \{c^Tx \mid Ax\leq b \}  \geq
\sup_{\lambda \in \RR^m_+} \{ -b^T \lambda \mid A^T \lambda=-c \},
\end{equation}
\end{corollary}
where $\RR^m_+:=\big\{(x_1,x_2,\cdots,x_m)\mid x_i\geq0,\ i=1,2,\cdots,m\big\}$. Equality in~\eqref{eq:fenchel_LP} holds if $b\in\ran A+\RR^m_+$. Moreover, both extrema are obtained when finite.

\begin{proof}
Take $f(x):=c^Tx$, $T:=A$ and $g(y):=\iota_{b_\geq}(y)$ where $b_\geq:=\{y\in \RR^m\mid y\leq b\}$. Then apply the polyhedral Fenchel duality Theorem~\ref{thm:Fenchel-duality-poly} observing that $f^*=\iota_{\{c\}}$, and for any $\lambda\in\RR^m$,
$$g^*(\lambda)=\sup_{y\leq b}y^T\lambda=\left\{\begin{array}{ll}
b^T \lambda, & \text{if } \lambda\in \RR^m_+;\\
+\infty,& \text{otherwise};
\end{array}\right.$$
and~\eqref{eq:fenchel_LP} follows, since $\dom g\cap A\dom f=\{Ax\in\RR^m\mid Ax\leq b\}$.
\end{proof}

One can easily derive various relevant results from Fenchel duality, such as the Sandwich theorem, the subdifferential sum rule, and the Hahn-Banach extension theorem, among many others.

\begin{theorem}[Extended sandwich theorem] \label{thm:sandwich-utility}
Let $X$ and $Y$ be
Banach spaces and let $T:X \to Y$ be a bounded linear mapping.
Suppose that
$f:X \rightarrow \RX$, $g:Y \rightarrow \RX$ are proper convex functions which together with $T$ satisfy either~\eqref{eq:Fen-dual-1} or~\eqref{eq:Fen-dual-2}. Assume that $f\geq-g\circ T$.
Then there is an affine function $\alpha: X \to \RR$ of the form
$\alpha(x) = \langle T^*y^*,x\rangle + r$ satisfying $f \ge \alpha \ge  - g \circ T$.
Moreover, for any $\bar x$ satisfying $f(\bar x) = (-g\circ T)(\bar x)$, we have
$-y^* \in \partial g(T\bar x)$.
\end{theorem}

\begin{proof}
With notation as in the Fenchel duality Theorem~\ref{thm:Fenchel-duality-fd}, we know $d = p$, and since $p \ge 0$ because $f(x) \ge - g(Tx)$, the supremum in $d$ is attained. Therefore there exists $y^* \in Y^*$ such that
$$0\leq p=d=-f^*(T^*y^*)-g^*(-y^*).$$
Then, by Fenchel-Young inequality~\eqref{eq:Fenchel-Young-ineq-fd}, we obtain
\begin{equation}\label{eq:proof_sandwich}
0 \le p \le f(x)-\langle T^*y^*,x\rangle + g(y)+\langle y^*,y\rangle,
\end{equation}
for any $x\in X$ and $y\in Y$.
For any $z \in X$, setting $y = Tz$ in the previous
inequality, we obtain
$$
a:=\sup_{z \in X} [-g(Tz) - \langle T^* y^*,z\rangle]\le
b:=\inf_{x \in X}[f(x) - \langle T^* y^*, x\rangle]
$$
Now choose $r \in [a,b]$. The affine function $\alpha(x) := \langle T^*
y^*,x\rangle + r$ satisfies $f \ge \alpha \ge  - g \circ T$, as claimed.

The last assertion  follows from~\eqref{eq:proof_sandwich} simply by setting $x=\bar x$, where $\bar x$ satisfies $f(\bar x) = (-g\circ T)(\bar x )$.  Then we have
$\sup_{y\in Y}\{\langle-y^*,y\rangle-g(y)\}\leq  (-g\circ T)(\bar x )-\langle T^*y^*,\bar x\rangle$.
Thus $g^*(-y^*)+ g(T\bar x)\leq-\langle y^*,T\bar x\rangle$ and hence $-y^* \in \partial g(T\bar x)$.
\end{proof}

When $X=Y$ and $T$ is the identity we recover the classical Sandwich theorem.
The next example shows that without a constraint qualification, the sandwich theorem may fail.

\begin{example}
Consider $f,g:\RR\to \RX$ given by
$$f(x):=\left\{\begin{array}{ll}
-\sqrt{-x},& \text{for } x\leq 0,\\
+\infty& \text{otherwise,}
\end{array}\right.\quad and \quad
g(x):=\left\{\begin{array}{ll}
-\sqrt{x},& \text{for } x\geq 0,\\
+\infty& \text{otherwise.}
\end{array}\right.$$
In this case, $\bigcup_{\lambda>0} \lambda\left[\dom g - \dom f\right]=\left[0,+\infty\right[\neq\RR$ and it is not difficult to prove there is not any affine function which separates $f$ and $-g$, see Figure~\ref{fig:sandwich}.\qede
\end{example}

The prior constraint qualifications are sufficient but not necessary for the sandwich theorem as we illustrate in the next example.

\begin{example}
Let $f,g:\RR\to \RX$ be given by
$$f(x):=\left\{\begin{array}{ll}
\frac{1}{x},& \text{for } x> 0,\\
+\infty& \text{otherwise,}
\end{array}\right.\quad and \quad
g(x):=\left\{\begin{array}{ll}
-\frac{1}{x},& \text{for } x< 0,\\
+\infty& \text{otherwise.}
\end{array}\right.$$
Despite that $\bigcup_{\lambda>0} \lambda\left[\dom g - \dom f\right]=\left]-\infty,0\right[\neq\RR$, the affine function $\alpha(x):=-x$ satisfies $f\geq\alpha\geq -g$, see Figure~\ref{fig:sandwich}.\qede
\end{example}

\begin{figure}[h!]
\centering
\includegraphics[height=.4\textwidth]{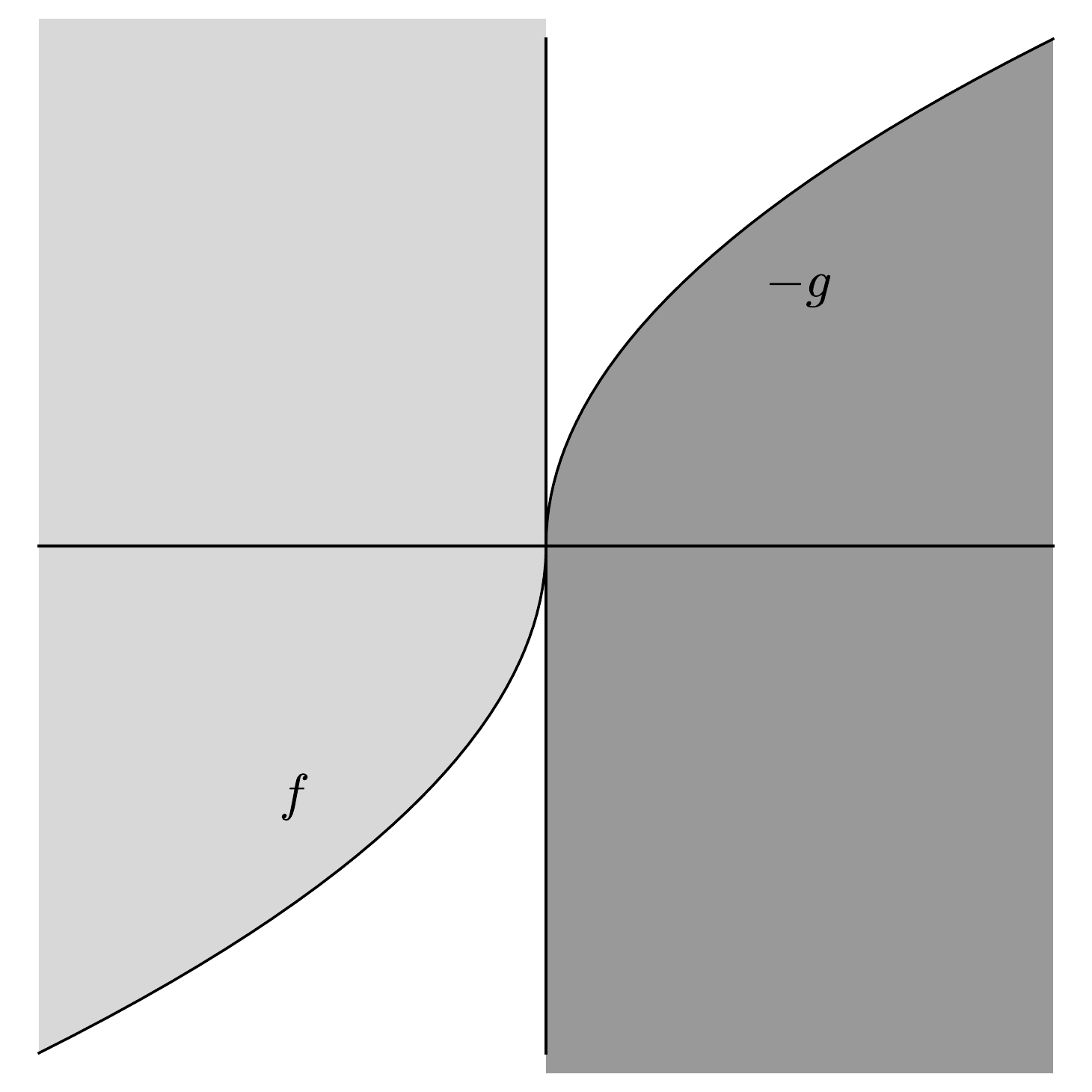}\qquad\includegraphics[height=.4\textwidth]{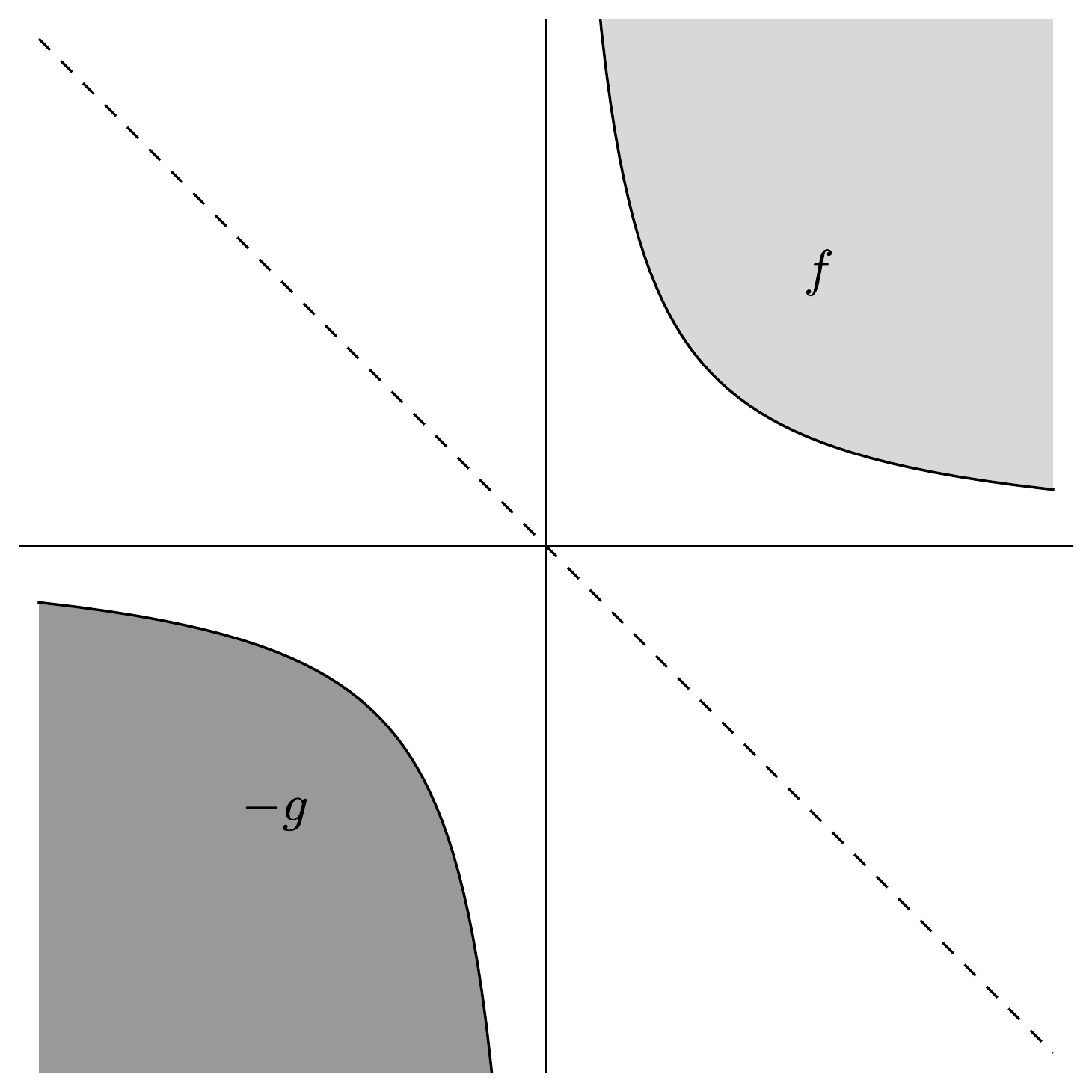}
\caption{On the left we show the failure of the sandwich theorem in the absence of the constraint qualification;
of the right we show that the constraint qualification is not necessary.}\label{fig:sandwich}
\end{figure}

\begin{theorem}[Subdifferential sum rule] \label{thm:sum-rule}
Let $X$ and $Y$ be Banach spaces, and let $f:X \to \RX$ and
$g:Y \to \RX$ be convex functions and let $T:X \to Y$ be a bounded
linear mapping. Then at any point $x \in X$ we have the sum rule
$$
\partial(f + g\circ T)(x) \supseteq \partial f(x) + T^*(\partial g(Tx))
$$
with equality if (\ref{eq:Fen-dual-1}) or (\ref{eq:Fen-dual-2}) hold.
\end{theorem}

\begin{proof}
The inclusion is straightforward by using the definition of the subdifferential, so we prove the reverse inclusion. Fix any $x\in X$ and let $x^*\in\partial(f+g\circ T)( x)$. Then
$0\in\partial(f-\langle x^* ,\cdot\,\rangle+g\circ T)( x)$. Conditions for the equality in Theorem~\ref{thm:Fenchel-duality-fd} are satisfied for the functions $f(\cdot)-\langle x^*,\cdot\,\rangle$ and $g$. Thus, there exists $y^*\in Y^*$ such that
$$f( x)-\langle x^*, x\rangle +g(T x)=-f^*(T^*y^*+x^*)-g^*(-y^*).$$
Now set $z^*:=T^*y^*+x^*$. Hence, by the Fenchel-Young inequality~\eqref{eq:Fenchel-Young-ineq-fd}, one has
$$0\leq f( x)+f^*(z^*)-\langle z^*, x\rangle=-g(T x)-g^*(-y^*)-\langle T^*y^*, x\rangle\leq0;$$
whence,
\begin{eqnarray*}
f( x)+f^*(z^*)=\langle z^*, x\rangle\\
g(T x)+g^*(-y^*)=\langle -y^*,T x\rangle.
\end{eqnarray*}
Therefore equality in Fenchel-Young occurs, and one has $z^*\in\partial f( x)$ and $-y^*\in\partial g(T x)$, which completes the proof.
\end{proof}

The subdifferential sum rule for two convex functions with a finite common point where one of them is continuous was proved by Rockafellar in 1966 with an argumentation based on Fenchel duality, see~\cite[Th.~3]{R66}. In an earlier work in 1963, Moreau~\cite{M63_2} proved the subdifferential sum rule for a pair of convex and lsc functions, in the case that infimal convolution of the conjugate functions is achieved, see~\cite[p.~63]{Mnotes} for more details. Moreau actually proved this result for functions which are the supremum of a family of affine continuous linear functions, a set which agrees with the convex and lsc functions when $X$ is a locally convex vector space, see~\cite{M62} or~\cite[p.~28]{Mnotes}. See also~\cite{HUMSV,HUP,BotWank1,BBY2} for more information about the subdifferential calculus rule.

\begin{theorem}[Hahn--Banach extension] \label{Hahn-Banach}
Let $X$ be a Banach space and
let $f \colon X \rightarrow \RR$ be a continuous sublinear function with $\dom f=X$.
Suppose that $L$ is a linear subspace of $X$
and the function $h \colon L \rightarrow \RR$ is
linear and {\em dominated} by $f$, that is, $f \geq h$ on $L$.
Then there exists
$x^*\in X^*$, dominated by $f$, such that
$$h(x)=\langle x^*,x\rangle,\text{~for all~} x\in L.$$
\end{theorem}

\begin{proof}
Take $g:=-h+\iota_L$ and apply Theorem~\ref{thm:Fenchel-duality-fd} to $f$ and $g$ with $T$ the identity mapping. Then, there exists $x^*\in X^*$
such that
\begin{align}
0&\leq\inf_{x\in X}\left\{f(x)-h(x)+\iota_L(x)\right\}\nonumber\\
&=-f^*(x^*)-\sup_{x\in X}\{\langle -x^*,x\rangle +h(x) -\iota_L(x)  \}\nonumber\\
&=-f^*(x^*)+\inf_{x\in L}\{\langle x^*,x\rangle-h(x)\};\label{eq:H-B}
\end{align}
whence,
$$f^*(x^*)\leq \langle x^*,x\rangle-h(x),\quad\text{for all } x\in L.$$
Observe that $f^*(x^*)\geq 0$ since $f(0)=0$. Thus, being $L$ a linear subspace, we deduce from the above inequality that
$$h(x)=\langle x^*,x\rangle,\quad\text{for all } x\in L.$$
Then~\eqref{eq:H-B} implies $f^*(x^*)=0$, from where
$$f(x)\geq\langle x^*,x\rangle ,\quad\text{for all } x\in X,$$
and we are done.
\end{proof}

\begin{remark}
(Moreau's max formula, Theorem~\ref{thm:max-formula})---a true child
of Cauchy's principle of steepest descent---can be also derived from
Fenchel duality. In fact, the non-emptiness of the subdifferential
at a point of continuity, Moreau's max formula, Fenchel duality, the
Sandwich theorem, the subdifferential sum rule, and Hahn-Banach
extension theorem are all equivalent, in the sense that they are
easily inter-derivable.

In outline, one considers $h(u):=\inf_x \big(f(x)+g(Ax+u)\big)$ and
checks that $\partial h(0) \neq \emptyset$ implies the Fenchel and
Lagrangian duality results; while condition \eqref{eq:Fen-dual-1} or
\eqref{eq:Fen-dual-2} implies $h$ is continuous at zero and thus
Theorem~\ref{thm:max-formula} finishes the proof. Likewise, the
polyhedral calculus \cite[\S5.1]{BorLew} implies $h$ is polyhedral
when $f$ and $g$ are and  shows that polyhedral functions have $\dom
h = \dom \partial h$. This establishes Theorem
\ref{thm:Fenchel-duality-poly}. This also recovers  abstract LP
duality (e.g., semidefinite programming and conic duality)  under
condition \eqref{eq:Fen-dual-1}. See~\cite{BorLew,BorVan} for more
details.\qede
\end{remark}

Let us turn to two illustrations of the power of convex analysis within functional analysis.

A \emph{Banach limit} is a bounded linear functional $\Lambda$ on the space of bounded sequences of real numbers $\ell^\infty$ such that
\begin{enumerate}
\item $\Lambda((x_{n+1})_{n\in\NN})=\Lambda((x_n)_{n\in\NN})$ (so it only depends on the sequence's tail),
\item $\liminf_k x_k \le \Lambda\big((x_n)_{n\in\NN}\big) \le \limsup_k x_k$
\end{enumerate}
where $(x_n)_{n\in\NN}=(x_1,x_2,\ldots)\in \ell^{\infty}$ and $(x_{n+1})_{n\in\NN}=(x_2,x_3,\ldots)$. Thus $\Lambda$ agrees with the limit on $c$, the subspace of sequences whose limit exists. Banach limits care peculiar objects!

The Hahn-Banach extension theorem can be used show the existence of Banach limits (see  Sucheston \cite{SUCH} or \cite[Exercise~5.4.12]{BorVan}). Many of its earliest applications were to summability theory and related fields.
We sketch Sucheston's proof as follows.

\begin{theorem}[Banach limits]\label{banal}
\emph{(See~\cite{SUCH}.)}
Banach limits exist.
\end{theorem}
\begin{proof}
Let $c$ be the subspace of convergent sequences in $\ell^{\infty}$.  Define $f:\ell^{\infty}\rightarrow\RR$ by
\begin{align}\label{eq:limit_Banach}
x:=(x_n)_{n\in\NN}\mapsto\lim_{n\rightarrow\infty}\left
(\sup_{j}\frac{1}{n}\sum_{i=1}^n x_{i+j}\right).
\end{align}
Then $f$ is sublinear with full domain, since the limit in~\eqref{eq:limit_Banach} always exists (see~\cite[p.~309]{SUCH}).  Define $h$ on $c$ by $h:=\lim_n x_n$ for every $x:=(x_n)_{n\in\NN}$ in $c$. Hence $h$ is linear  and agrees  with $f$ on $c$. Applying the Hahn-Banach extension Theorem~\ref{Hahn-Banach},
 there exists
$\Lambda\in (\ell^{\infty})^*$, dominated by $f$, such that
$\Lambda=h$ on $c$. Thus $\Lambda$ extends the limit  linearly from $c$ to $\ell^{\infty}$. Let $S$ denote the forward shift defined as $S((x_n)_{n\in\NN}):=(x_{n+1})_{n\in\NN}$. Note that $f(Sx-x)=0$, since
$$|f(Sx-x)|=\left|\lim_{n\to\infty}\left(\sup_{j}\frac{1}{n}(x_{j+n+1}-x_{j+1})\right)\right|\leq\lim_{n\to\infty}\frac{2}{n}\sup_j|x_j|=0.$$
Thus, $\Lambda(Sx)-\Lambda(x)=\Lambda(Sx-x) \le 0$,  and
$\Lambda(x)-\Lambda(Sx)=\Lambda(x-Sx)\leq f(x-Sx)=0;$
that is, $\Lambda$ is indeed a Banach limit.
\end{proof}

\begin{remark}
One of the referees kindly pointed out that in the proof of Theorem~\ref{banal}, the function $h$ can be simply defined by
$h:\{0\}\rightarrow\RR$ with $ h(0)=0$.

\end{remark}
\begin{theorem}[Principle of uniform boundedness]
\emph{(See (\cite[Example~1.4.8]{BorVan}.)}
\label{ex:pub}
Let  $Y$ be another Banach space and $T_\alpha \colon X \to Y$ for $\alpha \in \mathcal A$ be bounded linear operators. Assume  that $\sup_{\alpha \in A}\|T_\alpha(x)\| < +\infty$ for each $x$ in $X$. Then  $\sup_{\alpha \in A}\|T_\alpha\| <+\infty.$
 \end{theorem}

 \begin{proof}
  Define a function $f_A$ by $$f_A(x):=\sup_{\alpha \in
A}\|T_\alpha(x)\|$$ for each $x$ in $X$. Then,  as observed in
Fact~\ref{basic-prin}\ref{basic-prin:E1}, $f_A$ is convex. It is also lower semicontinuous since each
mapping $x \mapsto \|T_\alpha(x)\|$ is continuous. Hence $f_A$ is a finite, lower semicontinuous and convex (actually
sublinear)  function. Now Fact~\ref{basic-prin}\ref{basic-prin:E4} ensures $f_A$ is
continuous at the origin. Select $\varepsilon >0$ with $\sup\{
f_A(x) \mid\|x\| \le \varepsilon\} \le 1+f_A(0)=1$. It follows that
$$
\sup_{\alpha \in A}\|T_\alpha\|=\sup_{\alpha \in A}\frac{1}{\varepsilon}\sup_{\|x\| \le
\varepsilon}\|T_\alpha(x)\|=\frac{1}{\varepsilon}\sup_{\|x\| \le \varepsilon}\sup_{\alpha \in
A}\|T_\alpha(x)\| \le \frac{1}{\varepsilon}.$$
Thus, uniform boundedness is revealed to be continuity of $f_A$.
 \end{proof}

\section{The Chebyshev problem}\label{SecChev:1}

Let $C$ be a nonempty subset of $X$. We define the {\em nearest point mapping} by
\label{note:nearest-pt}
$$
P_C(x) := \{v \in C \mid \|v - x\| = \di_C(x)\}.
$$
A set $C$ is said to be a \emph{Chebyshev set} if $P_C(x)$ is a singleton for every
$x \in X$.
If $P_C(x) \ne\varnothing$ for every $x \in X$, then $C$ is said to be {\em proximal}; the term {\em proximinal}
is also used.

In 1961 Victor Klee~\cite{K61} posed the following fundamental question: Is every Chebyshev set in a Hilbert space
convex? At this stage, it is known that the answer is affirmative for weakly closed sets. In what follows we will present a proof
of this fact via convex duality. To this end, we will make use of the following fairly simple lemma.

\begin{lemma}\emph{(See \cite[Proposition~4.5.8]{BorVan}.)} \label{prop:reflexive-weakly-closed-Chebyshev}
Let $C$ be a weakly closed Chebyshev subset of a Hilbert space $H$. Then the nearest point mapping $P_C$ is continuous.
\end{lemma}

\begin{theorem}
Let $C$ be a nonempty weakly closed subset of a Hilbert space $H$. Then $C$ is convex if and only if $C$ is a Chebyshev set.
\end{theorem}

\begin{proof}
For the direct implication, we will begin by proving that $C$ is proximal. We can and do suppose that $0\in C$.
Pick any $x\in H$. Consider the convex and lsc functions $f(z):=-\langle x,z\rangle + \iota_{B_H}(z)$ and $g(z):=\sigma_C(z)$. Notice that $\bigcup_{\lambda>0} \lambda\left[\dom g - \dom f\right]=H$ (in fact $f$ is continuous at $0\in\dom f\cap\dom g$).
With the notation of Theorem~\ref{thm:Fenchel-duality-fd}, one has $p=d$, and the supremum of the dual problem is attained if finite. Since
$f^*(y)=\|x+y\|$ and $g^*(y)=\iota_{C}(y)$, as  $C$ is closed, the dual problem~\eqref{eq:Fenchel_Dual_P} takes the form
$$d=\sup_{y\in H} \{-\|x+y\|-\iota_C(-y)\}=-\di_C(x).$$
Choose any $c\in C$. Observe that $0\leq \di_C(x)\leq \|x-c\|$. Therefore the supremum must be attained, and $P_C(x)\neq\varnothing$. Uniqueness  follows easily from the convexity of $C$.

For the converse, consider the function $f:=\frac{1}{2}\|\cdot\|^2+\iota_C$. We first show that
\begin{equation}
\partial f^*(x) = \{P_C(x)\}, \ \mbox{for all}\ x \in H. \label{eq:Zal02-3.69}
\end{equation}
Indeed, for $x \in H$,
\begin{align*}
f^*(x) &= \sup_{y\in C}\left\{\langle x,y \rangle-\frac{1}{2}\langle y,y\rangle \right \}\\
&= \frac{1}{2}\langle x,x\rangle + \frac{1}{2}\sup_{y\in C}\left\{-\langle x,x\rangle +2\langle x,y \rangle-\langle y,y\rangle\right \}\\
&= \frac{1}{2}\|x\|^2-\frac{1}{2}\inf_{y\in C}\|x-y\|^2=\frac{1}{2}\|x\|^2-\frac{1}{2}d^2_C(x)\\
&=\frac{1}{2} \|x\|^2 - \frac{1}{2} \|x - P_C(x)\|^2  = \langle x,P_C(x) \rangle -\frac{1}{2} \|P_C(x)\|^2 \\
&= \langle x,P_C(x) \rangle - f(P_C(x)).
\end{align*}
Consequently,  by Proposition~\ref{prop:Fenchel-Young-fd}, $P_C(x) \in \partial f^*(x)$ for $x \in X$. Now suppose $y \in \partial f^*(x)$, and
define $x_n = x + \frac{1}{n} (y - P_C(x))$. Then $x_n \to x$, and hence
$P_C(x_n) \to P_C(x)$ by Lemma~\ref{prop:reflexive-weakly-closed-Chebyshev}. Using the
subdifferential inequality, we have
$$
0 \le \langle x_n - x, P_C(x_n) - y \rangle = \frac{1}{n} \langle y - P_C(x), P_C(x_n) - y \rangle.
$$
This now implies:
$$
0 \le \lim_{n \to \infty} \langle y - P_C(x), P_C(x_n) - y\rangle = - \|y - P_C(x)\|^2.
$$
Consequently, $y = P_C(x)$ and so (\ref{eq:Zal02-3.69}) is established.

Since $f^*$ is continuous and we just proved that $\partial f^*$ is a singleton, Proposition~\ref{GatPC} implies that $f^*$ is G\^ateaux differentiable. Now $-\infty < f^{**}(x)\leq f(x)=\frac{1}{2}\|x\|^2$ for all $x\in C$. Thus, $f^{**}$ is a proper function. One can easily check that $f$ is sequentially weakly lsc, $C$ being  weakly closed. Therefore, Theorem~\ref{th:convex_conj_suff} implies that $f$ is convex; whence, $\dom f=C$ must be convex.
\end{proof}

Observe that we have actually proved that \emph{every Chebyshev set with a continuous projection mapping is convex} (and closed).
We finish the section by recalling a simple but powerful ``hidden convexity" result.

\begin{remark}{(See \cite{BacBorw}.)}
Let $C$ be  a closed subset of a Hilbert space $H$. Then
there exists a continuous and convex function $f$ defined on $H$ such that
$\di_C^2(x) =\|x\|^2-f(x),\,\forall x\in H$.
Precisely, $f$ can be taken as $x\mapsto\sup_{c\in C}\{2\langle x,c\rangle-\|c\|^2\}$.
\end{remark}

\section{Monotone operator theory}\label{SecMon:1}
Let $A\colon X\rightrightarrows X^*$ be a \emph{set-valued operator}
(also known as a relation, point-to-set mapping or multifunction),
i.e., for every $x\in X$, $Ax\subseteq X^*$, and let $\gra A :=
\menge{(x,x^*)\in X\times X^*}{x^*\in Ax}$ be the \emph{graph} of
$A$. The \emph{domain} of $A$
  is $\dom A:= \menge{x\in X}{Ax\neq\varnothing}$ and
$\ran A:=A(X)$ is the \emph{range} of $A$. We say that $A$ is
\emph{monotone} if
\begin{equation}
\scal{x-y}{x^*-y^*}\geq 0,\quad \text{for all } (x,x^*),(y,y^*)\in\gra A,
\end{equation}
and \emph{maximally monotone} if $A$ is monotone and $A$ has no proper
monotone extension (in the sense of graph inclusion). Given $A$
monotone, we say that $(x,x^*)\in X\times X^*$ is
\emph{monotonically related to} $\gra A$ if
\begin{align*}
\langle x-y,x^*-y^*\rangle\geq0,\quad\text{for all } (y,y^*)\in\gra
A.\end{align*} Monotone operators have frequently shown themselves
to be a key class of objects in both modern Optimization and
Analysis; see, e.g., \cite{Bor1,Bor2,Bor3,BY2}, the books
\cite{BC2011,BorVan,BurIus,ph,Si,Si2,RW-1998,Zalinescu,Zeidler2A,Zeidler2B} and the references given therein.

Given sets $S\subseteq X$ and $D\subseteq X^*$, we define $S^\bot$
by $S^\bot := \{x^*\in X^*\mid\langle x^*,x\rangle= 0,\quad \forall
x\in S\}$ and $D_{\bot}$  by $D_\bot := \{x\in X\mid\langle
x,x^*\rangle= 0,\quad \forall x^*\in D\}$ \cite{PheSim}. Then the
\emph{adjoint} of $A$ is the operator $A^*:X^{**}\rightrightarrows
X^*$ such that
\begin{equation*}
\gra A^* :=
\menge{(x^{**},x^*)\in X^{**}\times X^*}{(x^*,-x^{**})\in(\gra A)^{\bot}}.
\end{equation*}
Note that the adjoint is always a \emph{linear relation}, i.e. its graph is a linear subspace.

The \emph{Fitzpatrick function} \cite{Fitz88} associated with an
operator $A$ is the function $F_A:X\times X^* \to \RX$ defined by
\begin{align}\label{defFA}F_A(x, x^*):= \sup_{(a,a^*)\in\gra A}\Big( \langle x,a^*\rangle+\langle a,x^*\rangle-\langle a,a^*\rangle\Big).
\end{align}
Fitzpatrick functions have been proved to be an important tool in modern
monotone operator theory. One of the main reasons is shown in the
following result.
\begin{fact}[Fitzpatrick]
\emph{(See (\cite[Propositions~3.2\&4.2, Theorem~3.4 and Corollary~3.9]{Fitz88}.)}
\label{f:Fitz}
Let $A\colon X\To X^*$ be monotone with $\dom A\neq\varnothing$.
Then $F_A$ is  proper lower semicontinuous in the norm $\times$ weak$^{*}$-topology $\omega (X^{*}, X)$, convex, and $F_A=\langle\cdot,\cdot\rangle$ on $\gra A$.
 Moreover, if $A$ is maximally monotone,
for every $(x,x^*)\in X\times X^*$, the inequality
$$\scal{x}{x^*}\leq F_A(x,x^*)\leq F^*_A(x^*,x)$$ is true,
and the first equality holds if and only if $(x,x^*)\in\gra A$.
\end{fact}

The next result is central to maximal monotone operator theory and algorithmic analysis. Originally it was proved by more extended direct methods than the concise convex analysis argument we present next.

\begin{theorem}[Local boundedness]\emph{(See \cite[Theorem~2.2.8]{ph}.)}
\label{pheps:11}Let $A:X\To X^*$ be  monotone with $\inte\dom A\neq\varnothing$.
Then $A$ is locally bounded at $x\in\inte\dom A$, i.e., there exist $\delta>0$ and $K>0$ such that
\begin{align*}\sup_{y^*\in Ay}\|y^*\|\leq K,\quad \forall y\in x+\delta B_X.
\end{align*}
\end{theorem}
\begin{proof} Let $x\in\inte\dom A$.
After translating the graphs if necessary,
we can and do suppose that $x=0$ and $(0,0)\in\gra A$.
Define $f:X\rightarrow\RX$ by
\begin{align*}
y\mapsto\sup_{(a,a^*)\in\gra A,\,\|a\|\leq1}\langle y-a, a^*\rangle.
\end{align*}
By Fact~\ref{basic-prin}\ref{basic-prin:E1},
$f$ is convex and lower semicontinuous.
Since $0\in\inte\dom A$, there exists $\delta_1>0$ such that
$\delta_1 B_X\subseteq\dom A$. Now we show that $
\delta_1 B_X\subseteq\dom f$.  Let $ y\in\delta_1 B_X$ and $y^*\in Ay$.
Thence, we have
\begin{align*}
&\langle y-a,y^*-a^*\rangle\geq0,\quad\forall (a,a^*)\in\gra A,\,\|a\|\leq1\\
&\Rightarrow \langle y-a,y^*\rangle\geq\langle y-a,a^*\rangle,\quad\forall (a,a^*)\in\gra A,\,\|a\|\leq1\\
&\Rightarrow +\infty >(\|y\|+1)\cdot\|y^*\|\geq\langle y-a,a^*\rangle,\quad\forall (a,a^*)\in\gra A,\,\|a\|\leq1\\
&\Rightarrow f(y)<+\infty\quad\Rightarrow y\in\dom f.
\end{align*}
Hence  $\delta_1 B_X\subseteq\dom f$ and thus $0\in\inte\dom f$. By
Fact~\ref{basic-prin}\ref{basic-prin:E4},
 there is $\delta>0$ with $\delta\leq\min\{\frac{1}{2},\frac{1}{2}\delta_1\}$  such that
 \begin{align*}
 f(y)\leq f(0)+1,\quad\forall y\in2\delta B_X.
 \end{align*}
 Now we show that $f(0)=0$.
 Since $(0,0)\in\gra A$, then $f(0)\geq0$. On the other hand, by the monotonicity of $A$,
 $\langle a, a^*\rangle=\langle a-0, a^*-0\rangle\geq0$ for every $(a,a^*)\in\gra A$.
 Then we have $f(0)=\sup_{(a,a^*)\in\gra A,\,\|a\|\leq1}\langle 0-a, a^*\rangle\leq0$.
 Thence $f(0)=0$.

 Thus,
 $$\langle y,a^*\rangle\leq \langle a, a^*\rangle+1,\quad\forall y\in2\delta B_X, (a,a^*)\in\gra A,\,
\|a\|\leq\delta,$$
whence, taking the supremum with $y\in2\delta B_X$,
\begin{align*}
&2\delta\|a^*\|\leq \|a\|\cdot\|a^*\|+1\leq \delta\|a^*\|+1,\quad \forall (a,a^*)\in\gra A,\,
a\in\delta B_X \\
&\Rightarrow\|a^*\|\leq\frac{1}{\delta},\quad\forall (a,a^*)\in\gra A,\,
a\in\delta B_X.
 \end{align*}
 Setting $K:=\frac{1}{\delta}$, we get the desired result.
 \end{proof}

Generalizations of Theorem~\ref{pheps:11}  can be found in \cite{Si2,BorFitz} and \cite[Lemma~4.1]{BY3}.
\subsection{Sum theorem and Minty surjectivity theorem}

In the early 1960s, Minty \cite{Min62} presented an important
characterization of maximally monotone operators in a Hilbert space;
which we now reestablish.
The proof we give of Theorem~\ref{MonMinty} is due to Simons and
Z{\u{a}}linescu \cite[Theorem~1.2]{SimZal}. We denote by $\Id$ the
\emph{identity mapping} from $H$ to $H$.
\begin{theorem}[Minty]\label{MonMinty}
Suppose that $H$ is a Hilbert space. Let $A:H\rightrightarrows H$ be
monotone.  Then $A$ is maximally monotone if and only if $\ran
(A+\Id)=H$.
\end{theorem}

\begin{proof}
``$\Rightarrow$":  Fix any $x^*_0\in H$, and let $B:H\rightrightarrows H$
be given by $\gra B:= \gra A-\{(0,x^*_0)\}$. Then $B$ is maximally
monotone. Define $F:H\times H\rightarrow\RX$ by
\begin{align}
(x, x^*)\mapsto F_B(x,x^*)+\frac{1}{2}\|x||^2+\frac{1}{2}\|x^*||^2.
\end{align}
Fact~\ref{f:Fitz} together with
Fact~\ref{basic-prin}\ref{basic-prin:E5} implies that $F$ is
coercive.  By \cite[Theorem~2.5.1(ii)]{Zalinescu}, $F$ has a
minimizer. Assume that $(z, z^*)\in H\times H$ is a minimizer of
$F$.  Then we have $(0,0)\in \partial F(z,z^*)$.  Thus, $ (0,0)\in
\partial F_B(z,z^*)+(z,z^*) $ and $(-z,-z^*)\in \partial F_B(z,z^*)$. Then
$$\big\langle (-z,-z^*), (b,b^*)-(z,z^*)\big\rangle\leq
F_B(b,b^*)-F_B(z,z^*),\quad\forall (b,b^*)\in\gra B,$$
and by Fact~\ref{f:Fitz},
$$ \big\langle (-z,-z^*), (b,b^*)-(z,z^*)\big\rangle\leq
\langle b,b^*\rangle-\langle z,z^*\rangle,\quad\forall
(b,b^*)\in\gra B;$$ that is,
\begin{equation}
0\leq
\langle b,b^*\rangle-\langle z,z^*\rangle+\langle z,b\rangle+\langle z^*,b^*\rangle-\|z\|^2-\|z^*\|^2,\quad\forall (b,b^*)\in\gra B.\label{Mint:1}
\end{equation}
Hence,
$$ \big\langle b+z^*, b^*+z\big\rangle=\langle b,b^*\rangle+\langle z,b\rangle+\langle z^*,b^*\rangle+\langle z,z^*\rangle\geq\|z+z^*\|^2\geq 0,\quad\forall (b,b^*)\in\gra B,$$
which implies that $(-z^*,-z)\in\gra B$, since $B$ is maximally
monotone. This combined with~\eqref{Mint:1} implies $0\leq -2
\langle z,z^*\rangle-\|z\|^2-\|z^*\|^2$. Then we have $z=-z^* $, and
$(z,-z)=(-z^*,-z)\in\gra B$, whence $(z,-z)+(0,x^*_0)\in\gra A$.
Therefor $x^*_0\in Az+z$, which implies $x^*_0\in\ran (A+\Id)$.

``$\Leftarrow$":   Let $(v,v^*)\in H\times H$ be monotonically
related to $\gra A$. Since $\ran (A+\Id)=H$, there exists
$(y,y^*)\in\gra A$ such that $v^*+v=y^*+y$. Then we have
\begin{align*}-\|v-y\|^2=
\big\langle v-y, y^*+y-v-y^*\big\rangle=
\big\langle v-y, v^*-y^*\big\rangle\geq0.
\end{align*}
Hence $v=y$, which also implies $v^*=y^*$.  Thus $(v,v^*)\in\gra A$, and therefore
$A$ is maximally monotone.
\end{proof}

\begin{remark}
The extension of Minty's theorem to reflexive spaces (in which case it asserts the surjectivity of $A+J_X$ for the normalized duality mapping $J_X$ defined below) was originally proved by Rockafellar. The proof given in \cite[Proposition~3.5.6, page~119]{BorVan} which uses Fenchel's duality theorem more directly than the one we gave here,  is only slightly more
complicated than that of Theorem~\ref{MonMinty}.
\end{remark}

Let $A$ and $B$ be maximally monotone operators from $X$ to $X^*$.
Clearly, the \emph{sum operator} $A+B\colon X\To X^*\colon x\mapsto
Ax+Bx: = \menge{a^*+b^*}{a^*\in Ax\;\text{and}\;b^*\in Bx}$ is
monotone. Rockafellar established the following  important result in
1970 \cite{Rock70}, the so-called ``sum theorem'': Suppose that $X$
is  reflexive. If $\dom A\cap\inte\dom B\neq\varnothing$, then $A+B$
is maximally monotone. We can weaken this constraint qualification
to be that $\bigcup_{\lambda>0} \lambda\left[\dom A-\dom B\right]$
is a closed subspace (see \cite{AttRiaThe,Si2,SiZ,BorVan,AlDa}).

We  turn to a
new proof of this generalized result. To this end, we need the
following fact along with the definition of the partial
inf-convolution. Given two real Banach spaces $X,Y$ and $F_1,
F_2\colon X\times Y\rightarrow\RX$, the \emph{partial
inf-convolution} $F_1\Box_2 F_2$ is the function defined on $X\times
Y$ by
\begin{equation*}F_1\Box_2 F_2\colon
(x,y)\mapsto \inf_{v\in Y} \big\{F_1(x,y-v)+F_2(x,v)\big\}.
\end{equation*}.

\begin{fact}[Simons and Z\u{a}linescu]
\emph{(See \cite[Theorem~4.2]{SiZ} or
\cite[Theorem~16.4(a)]{Si2}.)}\label{F4} Let $X, Y$ be real Banach
spaces and $F_1, F_2\colon X\times Y \to \RX$ be proper lower
semicontinuous and convex bifunctionals. Assume that for every
$(x,y)\in X\times Y$,
\begin{equation*}(F_1\Box_2 F_2)(x,y)>-\infty
\end{equation*}
and that  $\bigcup_{\lambda>0} \lambda\left[P_X\dom F_1-P_X\dom F_2\right]$
is a closed subspace of $X$. Then for every $(x^*,y^*)\in X^*\times Y^*$,
\begin{equation*}
(F_1\Box_2 F_2)^*(x^*,y^*)=\min_{u^*\in X^*}
\left\{F_1^*(x^*-u^*,y^*)+F_2^*(u^*,y^*)\right\}.
\end{equation*}\end{fact}

We  denote  by $J_X$ \emph{the duality map} from $X$ to $X^*$, which
will be simply written as $J$, i.e., the subdifferential of the
function $\tfrac{1}{2}\|\cdot\|^2$. Let  $F\colon X\times
Y\rightarrow\RX$ be a bifunctional defined on two real Banach
spaces. Following the notation by Penot \cite{Penot2} we set
\begin{equation}
F^\intercal\colon Y\times X\colon (y,x)\mapsto F(x,y).
\end{equation}

\begin{theorem}[Sum theorem]\label{MonSum1}
Suppose that $X$ is reflexive. Let $A, B:X\rightrightarrows X$ be
maximally monotone.  Assume that $\bigcup_{\lambda>0}
\lambda\left[\dom A-\dom B\right]$ is a closed subspace.   Then
$A+B$ is maximally monotone.
\end{theorem}

\begin{proof}Clearly, $A+B$ is monotone. Assume that $(z,z^*)\in X\times X^*$ is monotonically related to $\gra (A+B)$.

Let $F_1:=F_A\Box_2 F_B$, and $F_2:=F_1^{*\intercal}$.  By \cite[Lemma~5.8]{BWY3},
$\bigcup_{\lambda>0} \lambda\left[P_X(\dom F_A)-P_X(\dom F_B)\right]$
is a closed subspace.  Then Fact~\ref{F4} implies that
\begin{align}
F_1^*(x^*,x)=\min_{u^*\in X^*}
\left\{F_A^*(x^*-u^*,x)+F_B^*(u^*,x)\right\},\quad\text{for all } (x,x^*)\in X\times X^*.\label{monSum:e2}
\end{align}
Set $G:X\times X^*\rightarrow\RX$ by
\begin{align*}
(x,x^*)\mapsto F_2(x+z,x^*+z^*)-\langle x, z^*\rangle-\langle z, x^*\rangle
+\frac{1}{2}\|x\|^2+\frac{1}{2}\|x^*\|^2.
\end{align*}
Assume that $(x_0, x_0^*)\in X\times X^*$
is a minimizer of $G$. (\cite[Theorem~2.5.1(ii)]{Zalinescu} implies
that minimizers exist since $G$ is coercive). Then we have $(0,0)\in \partial G(x_0,
x_0^*)$.  Thus, there exists $v^*\in Jx_0, v\in J_{X^*}x_0^* $ such
that $(0,0)\in \partial F_2(x_0+z,x^*_0+z^*)+(v^*,v )+(-z^*,-z) $,
and then
\begin{align*}
(z^*-v^*,z-v)\in \partial F_2(x_0+z,x^*_0+z^*).
\end{align*}
Thence
\begin{align}
\Big\langle(z^*-v^*,z-v),(x_0+z,x^*_0+z^*)\Big\rangle =F_2(x_0+z,x^*_0+z^*)+F^*_2(z^*-v^*,z-v).\label{monSum:e5}
\end{align}
Fact~\ref{f:Fitz} and \eqref{monSum:e2} show that
\begin{align*}
F_2\geq\langle \cdot,\cdot\rangle,\quad F_2^{*\intercal}=\overline{F_1}\geq\langle \cdot,\cdot\rangle.
\end{align*}
Then by \eqref{monSum:e5},
\begin{align}
\Big\langle(z^*-v^*,z-v),(x_0+z,x^*_0+z^*)\Big\rangle &=F_2(x_0+z,x^*_0+z^*)+F^*_2(z^*-v^*,z-v)\nonumber\\
&\geq \big\langle x_0+z,x^*_0+z^*\big\rangle +\big\langle z^*-v^*,z-v\big\rangle.\label{monSum:e5b}
\end{align}
Thus, since $v^*\in Jx_0, v\in J_{X^*}x_0^*$,
\begin{align*}
0\leq\delta&:=\Big\langle(z^*-v^*,z-v),(x_0+z,x^*_0+z^*)\Big\rangle
-\big\langle x_0+z,x^*_0+z^*\big\rangle -\big\langle z^*-v^*,z-v\big\rangle \\
&=\big\langle -x_0-v,x^*_0+v^*\big\rangle = \langle -x_0,x^*_0\rangle-\langle x_0,v^*\rangle-\langle v,x^*_0\rangle -\langle v,v^*\rangle\\
&=\langle -x_0,x^*_0\rangle -\frac{1}{2}\|x_0^*\|^2-\frac{1}{2}\|x_0\|^2-\frac{1}{2}\|v^*\|^2-\frac{1}{2}\|v\|^2-\langle v,v^*\rangle,
\end{align*}
which implies
$$ \delta=0\quad\text{and}\quad\langle x_0,x^*_0\rangle +\frac{1}{2}\|x_0^*\|^2+\frac{1}{2}\|x_0\|^2=0;$$
that is,
\begin{equation}
\delta=0\quad\text{and}\quad x^*_0\in -Jx_0.\label{monSum:e6}
\end{equation}
Combining \eqref{monSum:e5b} and \eqref{monSum:e6}, we have
$F_2(x_0+z,x^*_0+z^*)=
\big\langle x_0+z,x^*_0+z^*\big\rangle$.
By \eqref{monSum:e2} and Fact~\ref{f:Fitz},
\begin{align}
(x_0+z,x^*_0+z^*)\in\gra (A+B).\label{monSum:e7}
\end{align}
Since $(z,z^*)$ is monotonically related to $\gra (A+B)$, it follows from \eqref{monSum:e7} that
$$\big\langle x_0,x^*_0\big\rangle=\big\langle x_0+z-z,x^*_0+z^*-z^*\big\rangle\geq0,$$
and then by \eqref{monSum:e6},
$$ -\|x_0\|^2=-\|x^*_0\|^2\geq0,
$$
whence $(x_0, x^*_0)=(0,0)$. Finally, by \eqref{monSum:e7}, one
deduces that $(z,z^*)\in \gra (A+B)$ and $A+B$ is maximally
monotone.
\end{proof}

It is still unknown whether the reflexivity condition can be omitted in Theorem \ref{MonSum1} though many partial results exist, see \cite{Bor2,Bor3} and \cite[\S9.7]{BorVan}.

\subsection{Autoconjugate functions}\label{sec:main}

Given  $F\colon X\times X^*\rightarrow\RX$, we say that $F$ is
\emph{autoconjugate} if $F=F^{*\intercal}$ on $X\times X^*$. We say
$F$ is a \emph{representer} for $\gra A$ if
\begin{equation}
\gra A=\big\{(x,x^*)\in X\times X^*\mid F(x,x ^*)=\langle x, x^*\rangle\big\}.
\end{equation}
Autoconjugate functions are the core of representer theory, which has been comprehensively studied in Optimization
and Partial Differential Equations
(see \cite{BW09,BWY3,PenotZ,Si2,BorVan,Ghou}).

Fitzpatrick posed the following question in
 \cite[Problem~5.5]{Fitz88}:

\begin{quotation}
\noindent \emph{If $A:X\rightrightarrows X^*$ is maximally monotone,
 does there necessarily exist an autoconjugate representer for A}?
 \end{quotation}
Bauschke and Wang gave an affirmative answer to the above question in reflexive spaces by construction of the function
$\mathcal{B}_A$ in Fact~\ref{GFF:1}. The first
construction of an autoconjugate representer for a maximally monotone
operator satisfying a mild constraint qualification  in a reflexive space  was provided by
Penot and Z\u{a}linescu in \cite{PenotZ}.
This naturally raises a question:
\begin{quotation}
\noindent Is $\mathcal{B}_A$ still
 an autoconjugate representer for a maximally monotone operator $A$ in a general Banach space?
 \end{quotation}
 We give a negative answer to the above question in Example~\ref{FPEX:au1}: in certain spaces, $\mathcal{B}_A$ fails to be
 autoconjugate.

\begin{fact}[Bauschke and Wang] \emph{(See \cite[Theorem~5.7]{BW09}.)}
\label{GFF:1} Suppose that $X$ is reflexive. Let $A\colon X\To X^*$
be maximally monotone. Then
\begin{align}
\mathcal{B}_A \colon X\times X^*&\to \RX\nonumber\\
 (x,x^*)&\mapsto
\inf_{(y,y^*)\in X\times X^* }\Big\{
\tfrac{1}{2}F_A(x+y,x^*+y^*)+
\tfrac{1}{2}F^{*\intercal}_A(x-y,x^*-y^*)+\tfrac{1}{2}\|y\|^2
+\tfrac{1}{2}\|y^*\|^2\Big\} \label{e:GFF:1}
\end{align}
is an autoconjugate representer for $A$.
\end{fact}

We will make use of the following result to prove Theorem~\ref{Theoauc:1} below.

\begin{fact}[Simons]\label{Sduu:1}\emph{(See \cite[Corollary~10.4]{Si2}.)}
Let $f_1, f_2, g\colon X\rightarrow \RX$ be proper convex. Assume that $g$ is continuous at a point of $\dom f_1
-\dom f_2$. Suppose that
\begin{align*}
h(x):=\inf_{z\in X}\left\{\tfrac{1}{2} f_1(x+z)+\tfrac{1}{2} f_2(x-z)+\tfrac{1}{4}g(2z)\right\}>-\infty,\quad\forall x\in X.
\end{align*}
Then
\begin{align*}
h^*(x^*)=\min_{z^*\in X^*}\left\{\tfrac{1}{2} f^*_1(x^*+z^*)+\tfrac{1}{2} f^*_2(x^*-z^*)+\tfrac{1}{4}g^*(-2z^*)\right\},\quad\forall x^*\in X^*.
\end{align*}
\end{fact}

Let $A:X\rightrightarrows X^*$ be a linear relation.
  We say that $A$ is
\emph{skew} if $\gra A \subseteq \gra (-A^*)$;
equivalently, if $\langle x,x^*\rangle=0,\; \forall (x,x^*)\in\gra A$.
Furthermore,
$A$ is \emph{symmetric} if $\gra A
\subseteq\gra A^*$; equivalently, if $\scal{x}{y^*}=\scal{y}{x^*}$,
$\forall (x,x^*),(y,y^*)\in\gra A$.
We define the \emph{symmetric part} and the \emph{skew part} of $A$ via
\begin{equation}
\label{Fee:1}
P := \thalb A + \thalb A^* \quad\text{and}\quad
S:= \thalb A - \thalb A^*,
\end{equation}
respectively. It is easy to check that $P$ is symmetric and that $S$
is skew.

\begin{fact}\emph{(See \cite[Theorem~3.7]{BBWY3}.)}
\label{PBABD:2}
Let $A: X^*\rightarrow X^{**}$ be linear and continuous.
 Assume that $\ran A \subseteq X$ and that there exists $e\in X^{**}\backslash X$ such that
\begin{align*}
\langle Ax^*,x^*\rangle=\langle e,x^*\rangle^2,\quad \forall x^*\in X^*.
\end{align*}
Let $ P$ and $S$ respectively  be the symmetric part and skew
 part of $A$.  Let $T:X\rightrightarrows X^*$  be defined by
\begin{align}\gra T&:=\big\{(-Sx^*,x^*)\mid x^*\in X^*, \langle e, x^*\rangle=0\big\}
=\big\{(-Ax^*,x^*)\mid x^*\in X^*, \langle e, x^*\rangle=0\big\}.\label{PBABA:a1}
\end{align}
Then the following hold.
\begin{enumerate}
\item\label{PBAB:em01}
$A$ is a maximally monotone operator on $X^*$ .
\item\label{PBAB:emmaz1}
$Px^*=\langle x^*,e\rangle e,\ \forall x^*\in X^*.$

\item\label{PBAB:em1}
 $T$
is maximally monotone and skew on $X$.

\item\label{PBAB:emma1}
$\gra  T^*=\{(Sx^*+re,x^*)\mid x^*\in X^*,\ r\in\RR\}$.

\item\label{PBAB:em2}
$F_T=\iota_C$, where
$
C:=\{(-Ax^*,x^*)\mid x^*\in X^*\}$.

\end{enumerate}

\end{fact}

\bigskip

We  next give  concrete examples of $A,T$ as in Fact~\ref{PBABD:2}.

\begin{example}[$c_0$]\label{FPEX:1}{(See \cite[Example~4.1]{BBWY3}.)}
 Let $X: = c_0$, with norm $\|\cdot\|_{\infty}$ so that
  $X^* = \ell^1$ with norm $\|\cdot\|_{1}$, and  $X^{**}=\ell^{\infty}$  with its second dual norm
$\|\cdot\|_{*}$ (i.e., $\|y\|_{*}:=\sup_{n\in\NN}|y_n|,\, \forall
  y:=(y_n)_{n\in\NN}\in \ell^{\infty}$).
Fix
$\alpha:=(\alpha_n)_{n\in\NN}\in\ell^{\infty}$ with $\limsup
\alpha_n\neq0$, and let
$A_{\alpha}:\ell^1\rightarrow\ell^{\infty}$ be defined  by
\begin{align}\label{def:Aa}
(A_{\alpha}x^*)_n:=\alpha^2_nx^*_n+2\sum_{i>n}\alpha_n \alpha_ix^*_i,
\quad \forall x^*=(x^*_n)_{n\in\NN}\in\ell^1.\end{align}
\allowdisplaybreaks Now let $ P_{\alpha}$ and $S_{\alpha}$
respectively
  be the symmetric part and skew
 part of $A_{\alpha}$.  Let $T_{\alpha}:c_{0}\rightrightarrows X^*$  be defined by
\begin{align}\gra T_{\alpha}&
:=\big\{(-S_{\alpha} x^*,x^*)\mid x^*\in X^*,
 \langle \alpha, x^*\rangle=0\big\}
=\big\{(-A_{\alpha} x^*,x^*)\mid x^*\in X^*,
 \langle \alpha, x^*\rangle=0\big\}\nonumber\\
&=\big\{\big((-\sum_{i>n}
\alpha_n \alpha_ix^*_i+\sum_{i<n}\alpha_n \alpha_ix^*_i)_{n\in\NN}, x^*\big)
\mid x^*\in X^*, \langle \alpha, x^*\rangle=0\big\}.\label{PBABA:Eac1}
\end{align}
Then
\begin{enumerate}
\item\label{BCCE:A01} $\langle A_{\alpha}x^*,x^*\rangle=\langle \alpha , x^*\rangle^2,
\quad \forall x^*=(x^*_n)_{n\in\NN}\in\ell^1$
and \eqref{PBABA:Eac1} is well defined.

\item \label{BCCE:SA01} $A_{\alpha}$ is a maximally monotone.

\item\label{BCCE:A1} $T_{\alpha}$
is a maximally monotone  operator.

\item \label{BCCE:A6}  Let $G:\ell^1\rightarrow\ell^{\infty} $
 be \emph{Gossez's operator} \cite{Gossez1} defined by
\begin{align*}
\big(G(x^*)\big)_n:=\sum_{i>n}x^*_i-
\sum_{i<n}x^*_i,\quad \forall(x^*_n)_{n\in\NN}\in\ell^1.
\end{align*}
Then $T_e: c_0\To\ell^1$ as defined by
\begin{align*}
\gra T_e:=\{(-G(x^*),x^*)\mid x^*\in\ell^1, \langle x^*, e\rangle=0\}
\end{align*}
is a maximally monotone operator, where $e:=(1,1,\ldots,1,\ldots)$.\qede
\end{enumerate}
\end{example}

We may now show that $ \mathcal{B}_T$ need not be autoconjugate.

\begin{theorem}\label{Theoauc:1}
Let $A: X^*\rightarrow X^{**}$ be linear and continuous.
 Assume that $\ran A \subseteq X$ and that there exists $e\in X^{**}\backslash X$ such that
 $\|e\|<\frac{1}{\sqrt{2}}$ and
\begin{align*}
\langle Ax^*,x^*\rangle=\langle e,x^*\rangle^2,\quad \forall x^*\in X^*.
\end{align*}
Let $ P$ and $S$ respectively  be the symmetric part and skew
 part of $A$.  Let $T, C$ be defined as in Fact~\ref{PBABD:2}.
 Then
 \begin{align*}
 \mathcal{B}_T(-Aa^*, a^*)>\mathcal{B}^*_T( a^*,-Aa^*),\quad\forall a^*\notin\{e\}_{\bot}.
 \end{align*}
In consequence, $ \mathcal{B}_T$ is not autoconjugate.
\end{theorem}

\begin{proof}
First we claim that
\begin{align}
\iota_C^{*\intercal}|_{X\times X^*}=\iota_{\gra T}.\label{autoNW:1}
\end{align}
Clearly, if we set $D:=\{(A^*x^*, x^*)\mid x^*\in X^*\}$, we have
\begin{align}
\iota^{*\intercal}_C=\sigma_C^{\intercal}=\iota_{C^{\bot}}^{\intercal}=\iota_D,\label{autoNW:2}
\end{align}
where in the second equality we use the fact that $C$ is a subspace. Additionally,
\begin{align}
A^*x^*\in X&\Leftrightarrow (S+P)^*x^*\in X\Leftrightarrow
S^*x^*+P^*x^*\in X\Leftrightarrow -Sx^*+Px^*\in X\nonumber\\
&\Leftrightarrow -Sx^*-Px^*+2Px^*\in X
\Leftrightarrow 2Px^*-Ax^*\in X
\Leftrightarrow Px^*\in X\quad\text{(since $\ran A\subseteq X$)}\nonumber\\
&\Leftrightarrow \langle x^*,e\rangle e\in X\quad\text{(by Fact~\ref{PBABD:2}\ref{PBAB:emmaz1})}\nonumber\\
&\Leftrightarrow \langle x^*,e\rangle=0\quad\text{(since $e\notin X$)}.\label{autoNW:3}
\end{align}
Observe that $Px^*=0$ for all $x^*\in\{e\}_{\bot}$ by Fact~\ref{PBABD:2}\ref{PBAB:emmaz1}. Thus, $A^* x^*=-Ax^*$ for all $x^*\in\{e\}_{\bot}$. Combining~\eqref{autoNW:2} and~\eqref{autoNW:3}, we have
\begin{align*}
\iota^{*\intercal}_C|_{X\times X^*}=\iota_{D\cap(X\times X^*)}=\iota_{\gra T},
\end{align*}
and hence \eqref{autoNW:1} holds.

Let $a^*\notin\{e\}_{\bot}$. Then $\langle a^*,e\rangle\neq0$.
Now we compute $\mathcal{B}_T(-Aa^*, a^*)$.
By Fact~\ref{PBABD:2}\ref{PBAB:em2} and \eqref{autoNW:1},
\begin{align}
&\mathcal{B}_T(-Aa^*, a^*)\nonumber\\
&=\inf_{(y,y^*)\in X\times X^* }\left\{
\iota_C(-Aa^*+y, a^*+y^*)+
\iota_{\gra T}(-Aa^*-y,a^*-y^*)+\tfrac{1}{2}\|y\|^2
+\tfrac{1}{2}\|y^*\|^2\right\}.\label{autoNW:4}
\end{align}
Thus
\begin{align}
\mathcal{B}_T(-Aa^*, a^*)
&=\inf_{y=-Ay^* }\left\{
\iota_{\gra T}(-Aa^*-y,a^*-y^*)+\tfrac{1}{2}\|y\|^2
+\tfrac{1}{2}\|y^*\|^2\right\}\nonumber\\
&=\inf_{y=-Ay^*,\,\langle a^*-y^*,e\rangle=0 }\left\{
\tfrac{1}{2}\|y\|^2
+\tfrac{1}{2}\|y^*\|^2\right\}=\inf_{\langle a^*-y^*,e\rangle=0 }\left\{
\tfrac{1}{2}\|Ay^*\|^2
+\tfrac{1}{2}\|y^*\|^2\right\}\nonumber\\
&\geq\inf_{\langle a^*-y^*,e\rangle=0 }\langle Ay^*,y^*\rangle=\inf_{\langle a^*-y^*,e\rangle=0 }\langle e,y^*\rangle^2\nonumber\\
&=\langle e,a^*\rangle^2.
\label{autoNW:5}
\end{align}
\allowdisplaybreaks
Next we will compute
$\mathcal{B}^*_T( a^*,-Aa^*)$.
By Fact~\ref{Sduu:1} and \eqref{autoNW:4}, we have
\begin{align}
&\mathcal{B}^*_T( a^*,-Aa^*)\nonumber\\
&=\min_{(y^*,y^{**})\in X^*\times X^{**} }\left\{
\frac{1}{2}\iota^*_C( a^*+y^*,-Aa^*+y^{**})+
\frac{1}{2}\iota^*_{\gra T}(a^*-y^*,-Aa^*-y^{**})+\tfrac{1}{2}\|y^{**}\|^2
+\tfrac{1}{2}\|y^*\|^2\right\}\nonumber\\
&=\min_{(y^*,y^{**})\in X^*\times X^{**} }\left\{
\iota_{D}( -Aa^*+y^{**}, a^*+y^*)+
\iota_{(\gra T)^{\bot}}(a^*-y^*,-Aa^*-y^{**})+\tfrac{1}{2}\|y^{**}\|^2
+\tfrac{1}{2}\|y^*\|^2\right\}\quad\text{(by \eqref{autoNW:2})}\nonumber\\
&\leq\iota_{D}( -Aa^*+2Pa^*, a^*)+
\iota_{(\gra T)^{\bot}}(a^*,-Aa^*-2Pa^*)+\tfrac{1}{2}\|2Pa^*\|^2\quad\text{(by  taking $y^*=0, y^{**}=2Pa^*$)}\nonumber\\
&=
\iota_{\gra (-T^*)}(-Aa^*-2Pa^*,a^*)+\tfrac{1}{2}\|2Pa^*\|^2\nonumber\\
&=\tfrac{1}{2}\|2Pa^*\|^2\quad\text{(by  Fact~\ref{PBABD:2}\ref{PBAB:emma1})}\nonumber\\
&=\tfrac{1}{2}\|2\langle a^*,e\rangle e\|^2\quad\text{(by  Fact~\ref{PBABD:2}\ref{PBAB:emmaz1})}\nonumber\\
&=2\langle a^*,e\rangle^2\| e\|^2.\nonumber
\end{align}
This inequality along with \eqref{autoNW:5}, $\langle e,
a^*\rangle\neq0$ and $\|e\|<\frac{1}{\sqrt{2}}$, yield
 \begin{align*}
 \mathcal{B}_T(-Aa^*, a^*)\geq \langle e,a^*\rangle^2>2\langle a^*,e\rangle^2\| e\|^2\geq\mathcal{B}^*_T( a^*,-Aa^*),\quad\forall a^*\notin\{e\}_{\bot}.
 \end{align*}
Hence $ \mathcal{B}_T$ is not autoconjugate.
\end{proof}

\begin{example}[Example~\ref{FPEX:1} revisited]\label{FPEX:au1}
 Let $X: = c_0$, with norm $\|\cdot\|_{\infty}$ so that
  $X^* = \ell^1$ with norm $\|\cdot\|_{1}$, and  $X^{**}=\ell^{\infty}$  with its second dual norm
$\|\cdot\|_{*}$.
Fix
$\alpha:=(\alpha_n)_{n\in\NN}\in\ell^{\infty}$ with $\limsup
\alpha_n\neq0$ and $\|\alpha\|_{*}<\frac{1}{\sqrt{2}}$, and let
$A_{\alpha}:\ell^1\rightarrow\ell^{\infty}$ be defined  by
\begin{align}\label{def:Aa}
(A_{\alpha}x^*)_n:=\alpha^2_nx^*_n+2\sum_{i>n}\alpha_n \alpha_ix^*_i,
\quad \forall x^*=(x^*_n)_{n\in\NN}\in\ell^1.\end{align}
\allowdisplaybreaks Now let $ P_{\alpha}$ and $S_{\alpha}$
respectively
  be the symmetric part and skew
 part of $A_{\alpha}$.  Let $T_{\alpha}:c_{0}\rightrightarrows X^*$  be defined by
\begin{align}\gra T_{\alpha}&
:=\big\{(-S_{\alpha} x^*,x^*)\mid x^*\in X^*,
 \langle \alpha, x^*\rangle=0\big\}
=\big\{(-A_{\alpha} x^*,x^*)\mid x^*\in X^*,
 \langle \alpha, x^*\rangle=0\big\}\nonumber\\
&=\big\{\big((-\sum_{i>n}
\alpha_n \alpha_ix^*_i+\sum_{i<n}\alpha_n \alpha_ix^*_i)_{n\in\NN}, x^*\big)
\mid x^*\in X^*, \langle \alpha, x^*\rangle=0\big\}.\label{PBABA:Ea1}
\end{align}
Then, by Example~\ref{FPEX:1} and
Theorem~\ref{Theoauc:1},
 \begin{align*}
 \mathcal{B}_{T_{\alpha}}(-Aa^*, a^*)>\mathcal{B}^*_{T_{\alpha}}( a^*,-Aa^*),\quad\forall a^*\notin\{e\}_{\bot}.
 \end{align*}
In consequence, $ \mathcal{B}_{T_{\alpha}}$ is not autoconjugate.\qede
\end{example}

The latter raises a very interesting question:
\begin{problem}
Is there a maximally monotone operator on some (resp. every)  non-reflexive Banach space that has no autoconjugate representer?
\end{problem}

\subsection{The Fitzpatrick function and differentiability}

 The \emph{Fitzpatrick function}
introduced in \cite{Fitz88} was discovered precisely to provide a
more transparent convex alternative to the earlier saddle function
construction due to Krauss \cite{BorVan}---we have not discussed
saddle-functions but they produce interesting maximally monotone
operators \cite[\S33 \& \S37]{Rock70}. At the time, Fitzpatrick's
interests were more centrally in the differentiation theory for
convex functions and monotone operators.

The search for results relating when a maximally monotone $T$ is
single-valued to differentiability of $F_T$ did not yield
fruit, and he put the function aside. This is still the one area
where to the best of our knowledge $F_T$ has proved of
very little help---in part because generic properties of $\dom
F_T$ and of $\dom(T)$ seem poorly related.

 That said,
monotone operators often provide efficient ways to prove
differentiability of convex functions. The discussion of Mignot's
theorem in\cite{BorVan}  is somewhat representative of how this works as is the treatment in \cite{ph}.
By contrast, as we have seen the Fitzpatrick  function and its
relatives now provide the easiest access to a gamut of solvability
and boundedness results.

\section{Other results}\label{Sec:Other}

\subsection{Renorming results: Asplund averaging}

Edgar Asplund \cite{Asplund} showed how to exploit convex analysis
to provide remarkable results on the existence of equivalent norms
with nice properties. Most optimizers are unaware of his lovely idea
which  we recast in the language of inf-convolution. Our development
is a reworking of that in Day~\cite{DayBook}. Let us start with
two equivalent norms $\|\cdot\|_1$ and $\|\cdot\|_2$ on a Banach
space $X$. We consider the quadratic forms $p_0:=\|\cdot\|_1^2/2$ and
$q_0:=\|\cdot\|_2^2/2$, and average  for $n\geq0$ by
\begin{align} \label{sumn} p_{n+1}(x):=\frac{p_n(x)+q_n(x)}{2}\mbox{~and~}
 q_{n+1}(x):=\frac{(p_n \Box q_n)(2x)}2.\end{align}
Let $C>0$ be such that $q_0\leq p_0\leq (1+C)q_0$.  By the
construction of $p_n$ and $q_n$, we have $q_n\leq p_n\leq
(1+4^{-n}C)q_n$ (\cite[Lemma]{Asplund}) and so
 the sequences $(p_n)_{n\in\NN}$, $(q_n)_{n\in\NN}$ converge to a common limit: a convex quadratic function $p$.

We shall  show  that the norm $\|\cdot\|_3:=\sqrt{2p}$ typically
inherits the good properties of both $\|\cdot\|_1$ and
$\|\cdot\|_2$. This is based on the following fairly straightforward result.

\begin{theorem}[Asplund]\emph{(See
\cite[Theorem~1]{Asplund}.)}\label{RenTh:1}
If either $p_0$ or $q_0$ is strictly  convex, so is $p$.
\end{theorem}

We  make a very simple application in the case that $X$ is reflexive.
In \cite{LiDES1}, Lindenstrauss showed that every reflexive Banach space has
an equivalent strictly convex norm. The reader may consult \cite[Chapter 4]{BorVan} for more general results. Now take $\|\cdot\|_1$ to be an equivalent strictly convex norm on $X$, and take
$\|\cdot\|_2$ to be
  an equivalent smooth norm with its dual norm  on $X^*$ strictly convex. Theorem~\ref{RenTh:1} shows that  $p$ is strictly convex. We note that by Corollary \ref{infc} and Fact \ref{lem:conjugate-convo} $$q_{n+1}^*(x^*):=\frac{q_n^*(x^*)+q_n(x^*)}{2}\mbox{~ and~}p_{n+1}^*(x^*):=\frac{(p_n^* \Box q_n^*)(2x^*)}2$$ so that Theorem \ref{RenTh:1} applies to $p_0^*$ and $q_0^*$. Hence
  $p^*$  is strictly convex (see also \cite[Proof of Corollary~1, page~111]{Diestel}).  Hence $\|\cdot\|_3(:=\sqrt{2p})$ and its dual norm ($:=\sqrt{2p^*}$)
    are equivalent strictly convex norms  on $X$ and $X^*$ respectively.

    Hence
$\|\cdot\|_3$ is an equivalent strictly convex  and smooth norm (since its dual is strictly convex). The existence of such a norm was one ingredient of Rockafellar's first proof of the Sum theorem.

\subsection{Resolvents of maximally monotone operators and connection with convex functions}

It is well known since Minty, Rockafellar, and Bertsekas-Eckstein
that in Hilbert spaces, monotone operators can be analyzed from the
alternative viewpoint of certain nonexpansive (and thus Lipschitz
continuous) mappings, more precisely, the so-called resolvents.
Given a Hilbert space $H$ and a set-valued operator $A\colon
H\rightrightarrows H$, the \emph{resolvent} of $A$ is
$$J_A:=(\Id+A)^{-1}.$$
The history of this notion goes back to Minty \cite{Min62} (in
Hilbert spaces) and Brezis, Crandall and Pazy \cite{BCP70} (in
Banach spaces). There exist more general notions of resolvents based
on different tools, such as the normalized duality mapping, the
Bregman distance or other maximally monotone operators (see
\cite{KT-2008,AB-1997,BWY-2010}). For more details on resolvents on
Hilbert spaces see \cite{BC2011}.

The Minty surjectivity theorem (Theorem \ref{MonMinty} \cite{Min62})
implies that a monotone operator is maximally monotone if and only if the
resolvent is single-valued with full domain. In fact, a classical
result due to Eckstein-Bertsekas \cite{EB-1992} says even more.
Recall that a mapping $T:H\to H$ is \emph{firmly nonexpansive} if
for all $x,y\in H$, $\|Tx-Ty\|\leq \langle Tx-ty,x-y\rangle$.

\begin{theorem}\label{thm:maxmon-firmnonexp}
Let $H$ be a Hilbert space. An operator $A\colon H\rightrightarrows
H$ is (maximal) monotone if and only if $J_A$ is firmly
nonexpansive (with full domain).
\end{theorem}

\begin{example}
Given a closed convex set $C\subseteq H$, the normal cone operator
of $C$, $N_C$, is a maximally monotone operator whose resolvent can
be proved to be the metric projection onto $C$. Therefore, Theorem
\ref{thm:maxmon-firmnonexp} implies the firm nonexpansivity of the
metric projection.\qede
\end{example}

In the particular case when $A$ is the subdifferential of a possibly
non-differentiable convex function in a Hilbert space, whose maximal
monotonicity was established by Moreau \cite{M65} (in Banach spaces
this is due to Rockafellar \cite{R-1970}, see also \cite{BZ05,BorVan}), the resolvent turns into the
proximal mapping in the following sense of Moreau. If $f:H\to \RX$ is a lower
semicontinuous convex function defined on a Hilbert space $H$, the
\emph{proximal or proximity} mapping is the operator
$\prox_f : H\to H$ defined by
$$\prox_f(x):=\underset{y\in H}{\textrm{argmin}}\left\{f(y)+\frac{1}{2}\|x-y\|^2\right\}.$$
This mapping is well-defined because $\textrm{prox}_f(x)$ exists and
is unique for all $x\in H$. Moreover, there exists the following
subdifferential characterization: $u = \textrm{prox}_f(x)$ if and
only if $x-u\in\partial f(u)$.

\emph{Moreau's decomposition} in terms of the proximal mapping is a
powerful nonlinear analysis tool in the Hilbert setting that has been
used in various areas of optimization and applied mathematics. Moreau
established his decomposition motivated by problems in unilateral
mechanics. It can be proved readily by using the conjugate and
subdifferential.

\begin{theorem}[Moreau decomposition]
Given a lower semicontinuous convex function $f:H\to \RX$, for all
$x\in H$,
$$x=\prox_f(x) + \prox_{f^*}(x).$$
\end{theorem}

\begin{example}
Note that for $f:=\iota_C$, with $C$ closed and convex, the proximal
mapping turns into the projection onto a closed and convex set $C$.
Therefore, this result generalizes the decomposition by orthogonal
projection on subspaces. In particular, if $K$ is a closed convex
cone (thus $\iota_K^*=\iota_{K^-}$, see Example
\ref{ex:Moreau-Rockafellar}), Moreau's decomposition provides a
characterization of the projection onto $K$:
\begin{center}
$x=y+z$ with $y\in K$, $z\in K^-$ and $\langle y,z\rangle=0$
$\Leftrightarrow $ $y=P_K x$ and $z=P_{K^-} x$.
\end{center}
This illustrates that in Hilbert space,  the Moreau decomposition can
be thought of as generalizing the decomposition into positive and
negative parts of a vector in a normed lattice \cite[\S6.7]{BorVan} to
an arbitrary convex cone.\qede
\end{example}

There is another notion associated to an operator $A$, which is
strongly related to the resolvent. That is the \emph{Yosida
approximation} of index $\lambda>0$ or the \emph{Yosida
$\lambda$-regularization}:
$$A_\lambda:=(\lambda \Id+A^{-1})^{-1} = \frac{1}{\lambda}(\Id-J_{\lambda A}).$$
If the operator $A$ is maximally monotone, so is the Yosida
approximation, and along with the resolvent they provide the
so-called \emph{Minty parametrization} of the graph of $A$ that is
Lipschitz continuous in both directions \cite{RW-1998}:
$$(J_{\lambda A}(z),A_\lambda(z))=(x,y) \Leftrightarrow z=x+y, (x,y)\in \gra A.$$

If $A=\partial f$ is the
subdifferential of a proper lower semicontinuous convex function
$f$, it turns out that the Yosida approximation of $A$ is the
gradient of the \emph{Moreau envelope} of $f$ $e_\lambda f$, defined
as the infimal convolution of $f$ and $\|\cdot\|^2/2\lambda$, that
is,
$$e_\lambda f(x):= f\,\Box\, \frac{\|\cdot\|^2}{2\lambda}=\inf_{y\in H}
\left\{f(y)+\frac{1}{2\lambda}\|x-y\|^2\right\}.$$ This justifies the
alternative term Moreau-Yosida approximation for the mapping
$(\partial f)_\lambda =(\lambda \Id+(\partial f)^{-1})^{-1}$.  This
allows to obtain a proof in Hilbert space of the connection between the convexity of the
function and the monotonicity of the subdifferential (see
\cite{RW-1998}):\emph{ a proper lower semicontinuous function is convex if
and only  its \emph{Clarke subdifferential} is monotone}.

It is worth mentioning that generally the role of the Moreau
envelope is to approximate the function, with a regularizing effect
since it is finite and continuous even though the function may not
be so. This behavior has very useful implications in convex and
variational analysis.

\subsection{Symbolic convex analysis}\label{ssec:CAS}

The thesis work of Hamilton \cite{BorHam} has provided a conceptual and effective framework (the SCAT \emph{Maple} software) for computing conjugates, subdifferentials  and infimal convolutions of functions of several variables. Key to this is the notion of \emph{iterated conjugation} (analogous to iterated integration) and a good data structure.

As a first example, with some care, the convex conjugate of the
function \[f:x \mapsto \log \left(\frac{\sinh \left(3\,x \right)}{  \sinh
 x } \right)\] can be symbolically nursed to obtain the
result
\begin{eqnarray*} g: y \mapsto \frac y2 \cdot \log \left(\frac{y+\sqrt{16-3y^2}}{4-2y}\right)+ \log \left(\frac{\sqrt{16-3y^2}-2}6\right),
\end{eqnarray*}
with domain $[-2,2]$.

Since the conjugate of $g$ is much more
easily computed to be $f$, this produces a symbolic computational
proof that $f$ and $g$ are convex and are mutually conjugate.

Similarly, \emph{Maple}  produces the conjugate of $x \mapsto \exp(\exp(x))$ as
$y \mapsto y\left(\log \left( y \right)-W \left( y \right)  -1/W \left( y \right)\right)$
in terms of the \emph{Lambert's W} function---the multi-valued inverse of $z \mapsto z e^z$. This function is unknown to most humans but is built into both \emph{Maple} and \emph{Mathematica}. Thus Maple knows that to order five
\[g(y)=-1+ \left(-1+\log  y  \right) y-{\frac {1}{2}}{y}^{2}+
{\frac {1}{3}}{y}^{3}-{\frac {3}{8}}{y}^{4}+O \left( {y}^{5} \right).
\]

Figure \ref{fig:expexp} shows the \emph{Maple}-computed conjugate after the \emph{SCAT} package is loaded:
\begin{figure}[h!]
\centering
\includegraphics[height=.66\textwidth]{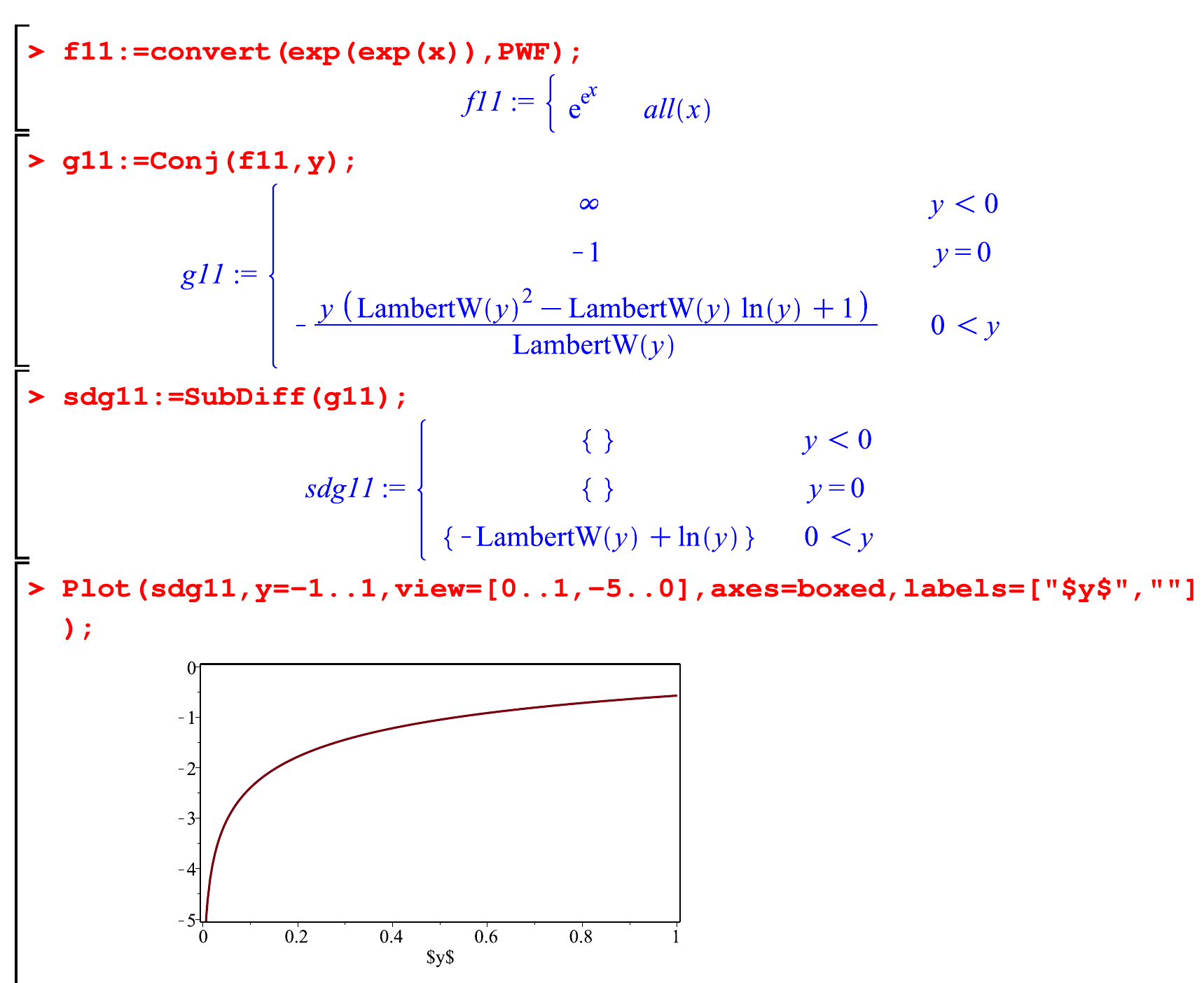}
\caption{The conjugate and subdifferential of $\exp\exp$.}\label{fig:expexp}
\end{figure}
There is a corresponding numerical program \emph{CCAT} \cite{BorHam}. Current work is adding the capacity to symbolically compute convex compositions---and so in principle Fenchel duality.

\subsection{Partial Fractions and Convexity}\label{partial-fractions}

We consider a network {\em objective function} $p_N$ given by
$$
p_N(q):=\sum_{\sigma \in S_N} \left(\prod_{i=1}^N
\frac{q_{\sigma(i)}}{\sum_{j=i}^N q_{\sigma(j)}}\right)
\left(\sum_{i=1}^N \frac{1}{\sum_{j=i}^N q_{\sigma(j)}}\right),
$$
summed over {\em all} $N!$ permutations; so a typical term is
$$
\left(\prod_{i=1}^N \frac{q_i}{\sum_{j=i}^N q_j}\right)
\left(\sum_{i=1}^N \frac{1}{\sum_{j=i}^n q_j}\right).
$$
For example, with $N=3$ this is
$$
q_1q_2q_3\left(\frac{1}{q_1+q_2+q_3}\right)\left(\frac{1}{q_2+q_3}\right)
 \left(\frac{1}{q_3}\right)
 \left(\frac{1}{q_1+q_2+q_3}+\frac{1}{q_2+q_3}+\frac{1}{q_3}\right).
$$
This arose as the objective function in research into coupon
collection.  The researcher, Ian Affleck,
wished to show $p_N$ was {\em convex} on the positive orthant.

First, we tried to simplify the expression for $p_N$.  The {\em
partial fraction decomposition} gives:
\begin{align}
p_1(x_1)&={\frac 1 {x_1}},  \label{par-simple}\\
p_2(x_1,x_2)&={\frac 1 {x_1}}+{\frac 1 {x_2}} - {\frac 1  {x_1+x_2}},\nonumber\\
p_3(x_1,x_2,x_3) &= {\frac 1 {x_1}}+{\frac 1 {x_2}}+{\frac 1{x_3}}
-{\frac 1{x_1+x_2}}-{\frac 1{x_2+x_3}}- {\frac 1{ x_1+x_3}}\nonumber
+{\frac 1{x_1+x_2+x_3}}.
\end{align}
In \cite{SenZik}, the simplified expression of $P_N$ is given by
\begin{align*}
p(x_1,x_2,\cdots,x_N)&:=
\sum_{i=1}^N\frac{1}{x_i}-\sum_{1\leq i<j\leq N}\frac{1}{x_i+x_j
}+\sum_{1\leq i<j<k \leq N}\frac{1}{x_i+x_j+x_k
}\\
&\quad-\ldots+(-1)^{N-1}\frac{1}{x_1+x_2+\ldots+x_N}.
\end{align*}
Partial fraction decompositions\index{partial fraction
decomposition} are another arena in which computer algebra systems
are hugely useful.  The reader is invited to try performing the
third case in \eqref{par-simple} by hand.  It is tempting to
predict the ``same'' pattern will hold for $N=4$.  This is easy to
confirm (by computer if not by hand) and so we are led to:

\begin{conj} For each $N\in\NN$, the function
\begin{eqnarray}\label{par-int} p_N(x_1,\cdots,x_N) =
\int_0^1\left(1-\prod_{i=1}^N (1-t^{x_i})\right) {\frac{dt} t}
\end{eqnarray}
is convex; indeed $1/p_N$ is concave.
\end{conj}

One may check  symbolically that this is true for $N<5$ via a
large Hessian computation.  But this is impractical for larger
$N$.  That said, it is easy to numerically sample the Hessian for
much larger $N$, and it is always positive definite.
Unfortunately, while the integral is convex, the integrand is not,
or we would be done.  Nonetheless, the process was already a
success, as the researcher was able to rederive his objective
function in the form of (\ref{par-int}).

A year after, Omar Hjab suggested re-expressing
(\ref{par-int}) as the {\em joint expectation}\, of Poisson
distributions.\footnote{ See ``Convex, II'' {\sl SIAM Electronic Problems and
Solutions} at \url{http://www.siam.org/journals/problems/downloadfiles/99-5sii.pdf}.}  Explicitly,
this leads to:

\begin{lemma} {\rm \cite[\S1.7]{BBG}} \label{joint-expectation}
 If $x=(x_1,\cdots,x_n)$ is a point in the positive orthant
$\RR_{++}^n,$ then
\begin{align}
\int_0^\infty\left(1-\prod_{i=1}^n(1-e^{-tx_i})\right)\,dt &=
\left(\prod_{i=1}^n x_i\right)\int_{\RR_{++}^n}e^{-\langle
x,y\rangle}\max(y_1,\cdots,y_n)\,dy, \nonumber \\[-6pt]
&\strut \label{j-exp}
\end{align}
where $\langle x,y\rangle=x_1y_1+\cdots+x_ny_n$ is the Euclidean
inner product.
\end{lemma}

It follows from the lemma---which is proven in \cite{BBG} with no recourse to probability theory---that
$$
p_N(x) =\int_{\RR_{++}^N}e^{-(y_1+\cdots+y_N)} \max\left(\frac
{y_1}{x_1},\cdots,\frac {y_N}{x_N}\right)\,dy,$$ and hence that
$p_N$ is positive, decreasing, and convex, as is the integrand. To
derive the stronger result that $1/p_N$ is concave we refer to \cite[\S1.7]{BBG}.
Observe that since $\frac{2ab}{a+b}\le\sqrt{ab}\le(a+b)/2$, it follows from
\eqref{j-exp} that $p_N$ is log-convex (and convex).  A
little more analysis of the integrand shows $p_N$ is strictly
convex on its domain. The same techniques apply when $x_k$ is replaced in \eqref{par-simple} or \eqref{par-int} by $g(x_k)$ for a concave positive function $g$.

Though much nice related work is to found in \cite{SenZik}, there is still no truly direct proof of the convexity of
$p_N$. Surely there should be! This development neatly shows both the power of computer assisted convex analysis and its current limitations.

\bigskip

Lest one think most results on the real line are easy, we challenge the reader to
prove the empirical observation that
$$p\mapsto \sqrt{p}\int_0^\infty\left|\frac{\sin x}{x}\right|^p\,d x$$
is \emph{difference convex} on $(1,\infty)$, i.e. it can be written as a difference of two convex functions \cite{BacBorw}.

\section{Concluding comments}

All researchers and practitioners in convex analysis and
optimization owe a great debt to Jean-Jacques Moreau---whether they
know so or not. We are delighted to help make his seminal role more
apparent to the current generation of scholars. For those who read
French we urge them to experience the pleasure of
\cite{M62,M63,M63_2,M65}  and especially \cite{Mnotes}. For others,
we highly recommend \cite{M66}, which follows \cite{M65} and of
which Zuhair Nashed wrote in his \emph{Mathematical Review}
MR0217617: ``There is a great need for papers of this kind; the
present paper serves as a model of clarity and motivation."

\paragraph{Acknowledgments}  The authors are grateful to
the three anonymous referees for their pertinent and constructive
comments. The authors also thank Dr.~Hristo S. Sendov for sending
them the manuscript \cite{SenZik}. The authors were all partially
supported by various Australian Research Council grants.

\end{document}